\documentclass[hidelinks,onefignum,onetabnum]{siamart220329}

\usepackage[utf8]{inputenc}
\usepackage{lipsum}
\usepackage{amsfonts}
\usepackage{graphicx}
\graphicspath{{img/}}
\usepackage{epstopdf}
\usepackage{algorithmic}
\ifpdf
  \DeclareGraphicsExtensions{.eps,.pdf,.png,.jpg}
\else
  \DeclareGraphicsExtensions{.eps}
\fi

\newsiamremark{remark}{Remark}
\newsiamremark{hypothesis}{Hypothesis}
\crefname{hypothesis}{Hypothesis}{Hypotheses}
\newsiamthm{claim}{Claim}
\newsiamthm{prop}{Proposition}

\headers{Improving efficiency of parallel accross the method SDC}{\v{C}aklovi\'c et al.}

\title{Improving efficiency of parallel across the method spectral deferred corrections\thanks{Submitted to the editors DATE.
\funding{This project has received funding from the European High-Performance Computing Joint Undertaking (JU) under grant agreement No 955701. The JU receives support from the European Union's Horizon 2020 research and innovation programme and Belgium, France, Germany, and Switzerland. This project also received funding from the German Federal Ministry of Education and Research (BMBF) grant 16HPC048.
This project has also received funding from the German Federal Ministry of Education and Research (BMBF) under grant 16ME0679K. Supported by the European Union - NextGenerationEU.}}}

\author{Gayatri \v{C}aklovi\'c\thanks{Karlsruhe Institute of Technology
  (\email{gayatri.caklovic@kit.edu}).}
\and Thibaut Lunet\thanks{Chair Computational Mathematics, Institute of Mathematics, Hamburg University of Technology, 21073 Hamburg, Germany
  (\email{thibaut.lunet@tuhh.de}, \email{sebastian.goetschel@tuhh.de}, \email{ruprecht@tuhh.de}).}
\and Sebastian Götschel\footnotemark[3]
\and Daniel Ruprecht\footnotemark[3]}

\usepackage{xspace}
\usepackage{stmaryrd}
\usepackage{amsopn}
\usepackage{dsfont}

\DeclareMathOperator{\diag}{diag}

\newcommand{\matr}[1]{\mathbf{#1}}
\newcommand{\vect}[1]{\boldsymbol{#1}}

\newcommand{\dt}{\Delta t}

\newcommand{\Imat}{\matr{I}}
\newcommand{\Qmat}{\matr{Q}}
\newcommand{\QDelta}{\matr{Q}_\Delta}

\newcommand{\RadauRight}{Radau-Right\xspace}
\newcommand{\Lobatto}{Lobatto\xspace}
\newcommand{\MINSRNS}{\texttt{MIN-SR-NS}\xspace}
\newcommand{\MINSRS}{\texttt{MIN-SR-S}\xspace}
\newcommand{\MINSRFLEX}{\texttt{MIN-SR-FLEX}\xspace}
\newcommand{\HOUWENSOMMEIJER}{\texttt{VDHS}\xspace}
\newcommand{\LU}{\texttt{LU}\xspace}
\newcommand{\IE}{\texttt{IE}\xspace}
\newcommand{\EE}{\texttt{EE}\xspace}
\newcommand{\IEpar}{\texttt{IEpar}\xspace}
\newcommand{\PIC}{\texttt{PIC}\xspace}

\newcommand{\I}{\mathbf{I}}

\newcommand{\Q}{\mathbf{Q}}
\newcommand{\Qd}{\Q_\Delta}
\newcommand{\Qdm}[1]{\Q_{\Delta, #1}}
\newcommand{\Qdi}{\Q_\Delta^{-1}}
\newcommand{\K}{\matr{K}}
\newcommand{\KS}{\K_{\rm{S}}}
\newcommand{\KNS}{\K_{\rm{NS}}}

\newcommand{\qdm}[1]{Q_{\Delta, #1}}

\newcommand{\eg}{\textit{e.g.},~}
\newcommand{\ie}{\textit{i.e.},~}

\newcommand{\C}{\mathbb{C}}
\newcommand{\R}{\mathbb{R}}
\newcommand{\node}{\tau}
\newcommand{\nodem}{\node_m}
\newcommand{\dl}{{\dt \lambda}}

\newcommand{\x}{\vect{x}}

\renewcommand{\u}{\vect{u}}
\newcommand{\uk}{\u^{k}}
\newcommand{\ukk}{\u^{k+1}}

\newcommand{\e}{\vect{e}}
\newcommand{\ek}{\e^{k}}
\newcommand{\ekk}{\e^{k + 1}}

\newcommand{\one}{\mathds{1}}

\ifdefined\captionText
\PackageError{notations}{variable(\string\captionText) is already defined.}
\fi

\definecolor{TLColor}{RGB}{175,0,175}

\definecolor{GCColor}{RGB}{17,149,0}

\definecolor{SGColor}{RGB}{54,0,189}

\definecolor{DRColor}{RGB}{0,148,189}

\definecolor{ModifColor}{RGB}{203,82,49}

\ifpdf
\hypersetup{
  pdftitle={Improving efficiency of parallel accross the method SDC},
  pdfauthor={G. \v{C}aklovi\'c et al.}
}
\fi

\begin{document}

\maketitle

% REQUIRED
\begin{abstract}
Parallel-across-the method time integration can provide small scale parallelism when solving initial value problems.
Spectral deferred corrections (SDC) with a diagonal sweeper, closely related to iterated Runge-Kutta methods proposed by Van der Houwen and Sommeijer, can use a number of threads equal to the number of quadrature nodes in the underlying collocation method.
However, convergence speed, efficiency and stability depend critically on the coefficients of the used SDC preconditioner.
Previous approaches used numerical optimization to find good diagonal coefficients.
Instead, we propose an approach that allows to find optimal diagonal coefficients analytically.
We show that the resulting parallel SDC methods provide stability domains and convergence order very similar to those of well established serial SDC variants.
Using a model for computational cost that assumes 80\% efficiency of an implementation of parallel SDC, we show that our variants are competitive with serial SDC and previously published parallel SDC coefficients as well as Picard iteration, a fourth-order explicit and a fourth-order implicit diagonally implicit Runge-Kutta method.
\end{abstract}

% REQUIRED

% REQUIRED

\tableofcontents
\section{Introduction}
Numerical methods to solve initial-value problems for nonlinear systems of ordinary differential equations (ODEs)
\begin{equation}
    \label{eq:ODEsys}
    \frac{du(t)}{dt} = f(t, u(t)), \quad
    t \in [0, T], \quad u(0) = u_0 \in \R^{N_{\text{dof}}},
\end{equation}
are of great importance for many domain sciences.
For ODEs arising from spatial discretization of a partial differential equation in a method-of-lines approach, the number of degrees of freedom $N_\text{dof}$ is often very large.
Hence, developing efficient methods to minimize time-to-solution and
computational cost becomes important.
Because of the large number of compute cores in modern computers,
leveraging concurrency is one of the most effective
ways to reduce solution times.

A widely used class of methods for solving~\eqref{eq:ODEsys} are Runge--Kutta Methods (RKM), usually represented by Butcher tables of the form
\begin{equation}\notag
    \begin{array}
    {c|c}
    \vect{c} & \matr{A} \\
    \hline\\[-1em]
    & \vect{b}^\top
\end{array}
\end{equation}
with $\matr{A} \in \R^{s \times s}$, $\vect{b}, \vect{c} \in \R^s$, and $s$ the number of stages.
The Butcher table is a concise way to represent the update
from $t_0$ to $t_0 + \dt$, which, for a scalar ODE, reads
\begin{align}
    \text{solve:} \quad& u(\vect{t}_s) - \dt \matr{A}
        f(\vect{t}_s, u(\vect{t}_s)) = u(t_0\vect{\one}),
        \quad \vect{t}_s = t_0\vect{\one} + \dt\vect{c},
    \label{eq:RKM_Solve}\\
    \text{update:} \quad& u(t_0+\dt)
        \approx u(t_0) + \dt \vect{b}^\top f(\vect{t}_s, u(\vect{t}_s)),
    \label{eq:RKM_Update}
\end{align}
where $\vect{\one}  = (1,\dots,1)^\top$ is the vector with unit entries and
$u(\vect{t}_s)$ is the vector containing solutions at all times in
$\vect{t}_s$.\footnote{For simplicity and to avoid any complex notation using tensor products, we describe time-integration here only from the scalar perspective.
}
For implicit Runge--Kutta methods (IRK), where  the matrix $\matr{A}$ is dense, computing the stages~\eqref{eq:RKM_Solve} requires solving a system of size $sN_\text{dof} \times sN_\text{dof}$.

A popular class of IRK are collocation methods~\cite[Sec~II.7]{hairer1993nonStiff}
based on Gaussian quadrature.
Collocation methods are attractive since they can be of very high order and are A-stable or L-stable depending on the used type of quadrature nodes.
However, since they have a dense matrix $\matr{A}$, they are computationally expensive.
Diagonally implicit Runge-Kutta methods (DIRK) with a lower-triangular $\matr{A}$ are computationally cheaper, because the implicit systems for the stages can be solved independently.
However, compared to collocation, DIRK methods have lower order for an equal number of stages and less favorable stability properties.

\subsection{Parallelism across the method}
First ideas to exploit parallelism in the numerical solution of ODEs emerged in the 1960s~\cite{nievergelt1964parallel}.
Given the massive increase in concurrency in modern high-performance computing,
the last two decades have seen a dramatic rise in interest in parallel-in-time methods,
see the recent reviews by Ong and Spiteri~\cite{OngEtAl2020} or Gander~\cite{Gander2015_Review}.
In the terminology established by Gear~\cite{Gear1988}, we focus on ``parallelism across the method'' where a time integration scheme is designed such that computations within a single time step can be performed in parallel.
By contrast, ``parallel across the steps'' methods like Parareal~\cite{LionsEtAl2001}, PFASST~\cite{EmmettMinion2012} or MGRIT~\cite{FalgoutEtAl2014_MGRIT} parallelize across multiple time steps.
Revisionist integral deferred corrections (RIDC) are a hybrid that compute a small number of time steps simultaneously~\cite{OngEtAl2016}.
While parallelism across the method is more limited in the number of cores it can employ, it often provides better parallel efficiency and is easier to implement than parallelization across the steps.

Note that time parallelization is meant to be employed in combination with spatial parallelization and not instead of it.
While we focus only on temporal parallelization here, parallel SDC has been shown to be capable of extending scaling beyond the saturation of pure spatial parallelization~\cite{freese2024parallel}.
It is in important to keep in mind that parallelization in space and in time are multiplicative.
That is, if a code's spatial parallelization saturates at, say, 1000 cores, adding parallelization in time with four cores would allow to use up to $4 \times 1000 = 4000$ cores.

For collocation methods, parallel across the method variants exist based on diagonalization of $\matr{A}$~
\cite{Butcher1976,Lie1987,NorsettEtAl1989,OREL1993241,speck2018parallelizing} or based on using a specific GMRES preconditioning~\cite{LevequeEtAl2023,Pazner2017700,MunchEtAl2023}.
However, those approaches can introduce significant overhead and, in particular for large problems, struggle to outperform sequential time-stepping~\cite{GayaThesis}.
Runge-Kutta methods with parallelism across stages or blocks of stages, that is with diagonal or lower block diagonal Butcher matrices, have also been considered but the resulting schemes lack stability and are of lower order than their sequential counterparts~\cite{IserlesNorsett1990, Jackson1991, JacksonEtAl1995, PETCU2001, SolodushkinEtAl2016}.

Spectral deferred corrections (SDC), introduced in 2000 by Dutt et al.~\cite{dutt2000spectral}, are an iterative approach for computing the stages of a collocation method by performing multiple ``sweeps'' through the quadrature nodes with a lower order method.
SDC can also be interpreted as a preconditioned  fixed-point or Richardson iteration~\cite{huang2006accelerating,qu2016numerical}.
In standard SDC, the preconditioner applied to~\eqref{eq:RKM_Solve} corresponds to a lower triangular
matrix that is inverted by forward substitution and results in a sweep-like type of iterations.
Speck~\cite{speck2018parallelizing} suggests the use of diagonal preconditioner instead, which allows to parallelize the iteration update for the stages.
However, depending on the entries of the diagonal preconditioner, convergence of parallel SDC can be much slower than convergence of standard SDC.
Unbeknownst to the author, this idea had been proposed before by van der Houwen \& Sommeijer~\cite{houwen1991iterated}, but in the context of iterated IRK methods instead of SDC.

Links between SDC and RKM are well established~\cite{christlieb2009comments}.
In addition, we show in \S\ref{sec:SDC} that SDC is actually equivalent to a specific iterated
IRK method by Houwen and Sommeijer~\cite{houwen1991iterated}
that would use a lower-triangular preconditioner.

IRK have been widely studied in the 1990s and have been shown to preserve important
attributes of the underlying collocation method such as order and stability~\cite{Burrage1993}.
They form the basis of ``Parallel iterated RK across the steps'' (PIRKAS) methods~\cite{RauberEtAl1996, VANDERHOUWEN1995309, VanderHouwen1990}, which can be written as Block Gauss-Seidel SDC (BGS-SDC) methods~\cite{baumann2042pursing, guibert2007parallel}.
Both PIRKAS and BGS-SDC methods have been combined with parallelism across the method using diagonal preconditioning~\cite{houwen1991iterated} to form the Parallel Diagonal-implicitly Iterated RK Across the Steps (PDIRKAS) methods~\cite{houwen1992embedded,VanderHouwen1993,Sommeijer1993,   VanderHouwen1994,VanderveenEtAl1995}, providing two levels of parallelism in time.
Similar two-level parallelism in time has been achieved by a combination of PFASST~\cite{EmmettMinion2012} with parallel SDC~\cite{SchoebelEtAl2020}.

The key to fast convergence and thus good performance of either parallel SDC or iterated IRK is the choice of coefficients in the diagonal preconditioner~\cite[Sec.~3]{houwen1991iterated}.
The authors identify two possible optimization problems to compute good diagonal coefficients for either non-stiff or stiff problems.
Both seek to minimize the spectral radius of certain matrices but, since these are ill conditioned, optimization algorithms struggle as the number of parallel stages $s$ increases.
This was already acknowledged both by van der Houwen \& Sommeijer and Speck~\cite{houwen1991iterated,speck2018parallelizing}.
Van der Houwen \& Sommeijer proposed to use an objective function based on the stability function of the iterated IRK as a remedy.
However, this is only applicable to some types of IRK methods and worked only for stiff problems.

\subsection{Contributions}
We present a generic approach to compute optimized coefficients for diagonal preconditioners in SDC or iterated IRK.
It leads to three sets of coefficients that we call \MINSRNS, \MINSRS and \MINSRFLEX.
While \MINSRNS is suited for non-stiff problems, \MINSRS and \MINSRFLEX are
designed for stiff problems.
We provide analytical expressions for the coefficients of \MINSRNS and
\MINSRFLEX, and a generic approach to generate \MINSRS coefficients for
any type of collocation method.
We show that the resulting parallel SDC methods are remain accurate and stable compared to state-of-the art SDC preconditioners, in particular those by van der Houwen \& Sommeijer~\cite{houwen1991iterated} and Weiser~\cite{weiser2015faster}.
We demonstrate that parallel SDC methods can compete standard RKM from the literature  with respect to computational cost against, because they allow to exploit parallelism across the method without sacrificing speed of convergence.
All numerical experiments reported in this paper can be reproduced with the accompanying code~\cite{speck2024pySDC}.

\section{Optimal diagonally preconditioned Spectral Deferred Corrections}
\label{sec:theory}
We start by describing SDC in \S\ref{sec:SDC} and develop the new diagonal preconditioners in \S\ref{sec:optcoeff}.
For details on SDC see Dutt et al. or Huang et al.~\cite{dutt2000spectral,huang2006accelerating}.

\subsection{Spectral Deferred Corrections as a fixed point iteration}
\label{sec:SDC}
Consider the Picard formulation of the initial value problem~\eqref{eq:ODEsys}
on $[t_0, t_0+\dt]$
\begin{equation}\label{eq:picard}
    u(t) = u_0 + \int_{t_0}^{t}f(s, u(s))ds.
\end{equation}
for some time step size $\dt$.
By choosing $M\in\mathbb{N}$ collocation nodes
$0 \leq \node_1 < \dots < \node_M \leq 1$ and defining $t_m = t_0+\node_m\dt$
we can write \eqref{eq:picard} for each $t_m$ as
\begin{equation}\label{eq:picardNodes}
    u(t_m) =
    u_0 + \dt\int_0^{\nodem}f(t_0+s\dt, u(t_0+s\dt))ds,
    \quad m = 1, \dots, M.
\end{equation}
Let $\ell_i$ denote the $i^{\rm{th}}$ Lagrange polynomial associated with
the nodes $\node_1, \dots, \node_M$.
Using a polynomial approximation of the integrand $f$ turns~\eqref{eq:picardNodes} into
\begin{equation}
    \label{eq:collocation}
    u_m = u_0 +
        \dt\sum_{i=1}^M \left(\int_0^{\node_m}\ell_i(s)ds \right)
        f(t_i, u_i), \quad m = 1, \dots, M,
\end{equation}
where $u_m$ is a discrete approximation of $u(t_m)$.
We collect~\eqref{eq:collocation} for $m=1,\ldots,M$ in the compact matrix formulation
\begin{equation}\label{eq:linear_collocation_problem}
    \u - \dt\Qmat f(\u) = u_0 \vect{\one},
\end{equation}
with $\u=[u_1,\dots,u_M]^\top$,
$f(\u) = [f(t_1, u_1), \dots, f(t_M, u_M)]^\top$,
$\vect{\one} = [1,\dots,1]^\top $.
The collocation matrix $\Qmat \in \R^{M \times M}$ has entries
\begin{equation}\label{eq:definition_Q}
    [\Qmat]_{ij} = \int_0^{\node_i}\ell_j(s)ds \; =: q_{i,j}.
\end{equation}
We refer to~\eqref{eq:linear_collocation_problem} as the collocation problem.
This is equivalent to the first RKM step~\eqref{eq:RKM_Solve} since the collocation method is an IRK method with
$\matr{A}=\matr{Q}$,
$[\vect{b}]_i = \int_{0}^{1}\ell_i(s)ds$
and $[\vect{c}]_i = \node_i$~\cite[Th.~7.7]{hairer1993nonStiff}.
The second RKM step~\eqref{eq:RKM_Update} is equivalent to the update
\begin{equation}\label{eq:collUpdate}
    u(t_0 + \Delta t) \approx u_0 + \dt\vect{b}^\top f(\u).
\end{equation}
Note that if $\node_{M} = 1$ we can also set $u(t_0 + \Delta t) \approx u_M$ instead.

\begin{remark}\label{rem:collUpdate}
    Using the update~\eqref{eq:collUpdate} improves the order of the step solution but can also reduce numerical stability~\cite[Rem.~4]{ruprecht2016spectral}.
    This is confirmed by numerical experiments not presented here, as some preconditioners that appear to be $A$-stable loose this property
    when performing the collocation update.
    While this aspect would need further investigation, it motivates
    us to only use node distributions with $\node_{M}=1$ (\ie \RadauRight and \Lobatto) for better numerical stability.
\end{remark}

The distribution and type of quadrature nodes controls the order of the collocation method~\cite[Tab.~2.1]{houwen1991iterated}.
We use Legendre polynomials but other types can be used as well.
The interested reader can consult the book by Gautschi~\cite{gautschi2004orthogonal} for details and examples.
For example, \Lobatto-type nodes for a Legendre distribution produce a method of order $2M-2$.
Radau-II-type nodes for a Legendre distribution produce a method of order $2M-1$.
We refer to Radau-II quadrature as \RadauRight to make explicit that it includes the right boundary node $\node_M=1$.

SDC solves~\eqref{eq:linear_collocation_problem} with the preconditioned
fixed-point iteration
\begin{equation}
    \label{eq:sweep}
    \u^{k+1} = \u^{k} + \matr{P}^{-1}\left[
        u_0 \vect{\one} - \left(\u^{k} -\dt\Qmat f(\u^{k})\right) \right],
\end{equation}
where $\matr{P[\u]}:= \u - \dt\QDelta f(\u)$ and $\QDelta \in \R^{M \times M}$ is a matrix called the SDC preconditioner.
Because $\QDelta$ is typically chosen to be lower triangular, the inversion of $\matr{P}$ can be computed node by node by forward substitution.
Therefore, one SDC iteration~\eqref{eq:sweep} is often called a ``sweep''.
It is fully defined by the used preconditioner $\matr{P}$.
Note that the \emph{SDC preconditioner} $\QDelta$ is problem independent in contrast to $\matr{P}$.
The generic form of an SDC sweep is
\begin{equation}
    \label{eq:sweepGeneric}
    \u^{k+1} - \dt\QDelta f(\u^{k+1}) =
    u_0 \vect{\one} + \dt(\Qmat-\QDelta) f(\u^{k}).
\end{equation}
By setting $\QDelta=0$ we retrieve the classical Picard iteration (\PIC).
The original SDC method~\cite{dutt2000spectral} considers $\QDelta$ matrices based on explicit (\EE) and implicit Euler (\IE)
\begin{equation}\notag
    \Qd^{\EE} =
    \begin{bmatrix}
        0 & 0 & \dots & 0 & 0\\
        \Delta \node_2 & 0 & \dots & 0 & 0\\
        \vdots & \vdots & \ddots & \vdots & \vdots \\
        \Delta \node_2 & \Delta \node_3 & \dots & \Delta \node_{M} & 0
    \end{bmatrix}, \quad
    \Qd^{\IE} =
    \begin{bmatrix}
        \Delta \node_1 & 0 & \dots & 0 \\
        \Delta \node_1 & \Delta \node_2 & \dots & 0 \\
        \vdots & \vdots & \ddots & \vdots \\
        \Delta \node_1 & \Delta \node_2 & \dots & \Delta \node_{M}
    \end{bmatrix},
\end{equation}
where $\Delta \node_m = \nodem - \node_{m-1}$ for $1 \leq m \leq M$,
and $\node_0=0$.
The computational cost of solving~\eqref{eq:sweepGeneric} with forward substitutions is the same as that of a  DIRK method with $s=M$ stages.

\begin{remark}
    \label{rem:sdcRK}
    We can also define the SDC sweep based on the $\matr{A}$ and $\vect{c}$
    Butcher arrays of a DIRK method.
    Setting $\QDelta=\Qmat=\matr{A}$ retrieves~\eqref{eq:RKM_Solve} from~\eqref{eq:sweepGeneric}, independent of $\u^0$.
    The collocation update~\eqref{eq:collUpdate} is equivalent to the RKM update~\eqref{eq:RKM_Update}.
    Hence, any generic SDC implementation based on~\eqref{eq:sweepGeneric}
    can be used to run any type of DIRK or explicit RK method.
    Such an approach is implemented in \texttt{pySDC}~\cite{speck2024pySDC}.
\end{remark}

As first suggested by Speck~\cite{speck2018parallelizing}, it is also possible to use a diagonal $\QDelta$.
This allows to compute the sweep update for all nodes in parallel using $M$ threads or processes.
Note that if we consider any IRK method introduced by van der Houwen \& Sommeijer~\cite[Eq.~3.1a]{houwen1991iterated} with a dense Butcher matrix $\matr{A}$, the diagonal preconditioned iteration in~\eqref{eq:RKM_Solve} is equivalent to the generic SDC sweep~\eqref{eq:sweepGeneric} with a diagonal $\QDelta$.

Possible choices for the entries of $\QDelta$ that have been suggested are the diagonal elements of $\Qmat$, that is  $\QDelta=\diag(q_{11},\dots,q_{MM})$, or $\QDelta=\diag(\tau_1,\dots,\tau_M)$, corresponding to an implicit Euler step from $t_0$ to $t_m$ (\IEpar).
However, these preconditioners result in slow convergence of the SDC iteration~\eqref{eq:sweep}, making it inefficient~\cite{speck2018parallelizing,houwen1991iterated}.
Hence, we focus on finding better coefficients for a diagonal $\QDelta$ for which the resulting SDC iteration converges rapidly.

\subsection{Optimal coefficients for diagonal preconditioning}
\label{sec:optcoeff}

Like van der Hou\-wen \& Sommeijer and Speck~\cite{houwen1991iterated,speck2018parallelizing} we start with Dahlquist's
test equation
\begin{equation}
    \label{eq:dahlquist}
    \frac{du}{dt} = \lambda u, \quad \lambda \in \mathbb{C}, \; t \in [0, T], \; u(0)=1.
\end{equation}
Applying~\eqref{eq:sweepGeneric} to~\eqref{eq:dahlquist} results in the sweep
\begin{equation}\label{eq:SDC_test_eq}
    (\I - \dl \Qd) \ukk = \dl (\Q - \Qd)\uk + \u_0.
\end{equation}
Let $\ek := \uk - \u$ be the error to the exact solution of the collocation problem \eqref{eq:linear_collocation_problem} and $z:=\dl$.
The iteration matrix $\K(z)$ governing the error is
\begin{equation}
    \label{eq:error_iterations}
    \ekk = \K(z) \ek, \quad \K(z) = z \left(\I - z \Qd\right)^{-1}(\Q - \Qd).
\end{equation}
where (NS) denotes the \emph{non-stiff} and (S) the \emph{stiff} limit of the SDC iteration matrix.

To find optimal diagonal coefficients for $\QDelta$, we use the spectral radius of the iteration matrix $\rho(\K(z))$ as indicator for convergence speed.
However, the dependency on $z=\dl$ would make the resulting optimization problem specific.
Therefore, we consider the spectral radii of the matrices
\begin{equation}\notag
	\KNS = \lim\limits_{|z| \rightarrow 0} \frac{\K(z)}{z}, \quad
    \KS = \lim\limits_{|z| \rightarrow \infty} \K(z).
\end{equation}
In particular, we aim to find diagonal coefficients such that they become
nilpotent to ensure fast asymptotic convergence~\cite{weiser2015faster}.
Short algebraic calculations yield
\begin{align}
    \KNS &= \Qmat-\QDelta,\\
    \KS &= \I-\QDelta^{-1}\Qmat.
\end{align}
\begin{table}\label{tab:coeffs}
    \centering
    \begin{tabular}{c|c|c}
        coefficients & reference & spectral radius $\rho(\KS)$ \\
        \hline
        \HOUWENSOMMEIJER & van der Houwen \& Sommeijer 1991~\cite{houwen1991iterated} & 0.025 \\
        \texttt{MIN} & Speck 2018~\cite{speck2018parallelizing} & 0.42 \\
        \texttt{MIN3} & Speck at al.~2024~\cite{speck2024pySDC} & 0.0081 \\ \hline
    \end{tabular}
    \caption{Spectral radius $\rho(\KNS)$ of the non-stiff iteration matrix for optimal diagonal coefficients found in the literature using $M=4$ \RadauRight nodes.}
\end{table}

The $\KS$ matrix was considered both by van der Houwen \& Sommeijer~\cite{houwen1991iterated} and Speck~\cite{speck2018parallelizing}.
Both noticed the difficulty of finding optimal coefficients when minimizing
the spectral radius of $\KS$.
Several diagonal coefficients are available in the literature and summarized in Appendix~\ref{ap:coeffs}.
Table~\ref{tab:coeffs} shows the spectral radius of the resulting $\KS$.
Values for $\rho(\KS)$ depend on the used optimization approach.
The \texttt{MIN3} coefficients proposed by Speck are similar to \HOUWENSOMMEIJER, but were obtained differently by using an online black box optimization software that unfortunately is not available anymore.

The matrix in the non-stiff limit $\KNS$ was only considered by van der Houwen \& Sommeijer~\cite{houwen1991iterated} but discarded because of the difficulty to numerically optimize the spectral radius of $\KNS$ and the poor performance of the obtained diagonal coefficients.
Both $\KS$ and $\KNS$ become very poorly conditioned as $M$ increases, which makes computing the optimal coefficients numerically very difficult.
By contrast, we propose an analytical approach to find diagonal coefficients by focusing on nilpotency of $\KNS$ and $\KS$ instead.

\subsubsection{Preliminaries}
Given a set of distinct nodes $0 \leq \node_1 < \dots < \node_M \leq 1$, we can associate any vector $\x \in \R^M$ uniquely with a polynomial $x \in P_{M-1}$ with real coefficients via the mapping
\begin{equation}\label{eq:Phi}
    \Phi: \x \mapsto x(t) = \sum_{j = 1}^M \x_j \ell_j(t),
\end{equation}
where $\ell_i \in P_{M-1}$ are the Lagrange polynomials for the nodes $\node_1, \dots, \node_M$.
Inversely, a polynomial $x \in P_{M-1}$ can be mapped to a vector $\x \in \R^M$ by evaluating it at the nodes and setting $\x_j = x(\node_j)$.
The bijective mapping $\Phi$ defines the isomorphism $\R^M \cong P_{M-1}$.
Using the definition of the collocation matrix $\Q$ from \eqref{eq:definition_Q}, we obtain
\begin{equation}\label{eq:Qx}
    \sum_{j = 1}^M q_{m,j} \x_j =
    \sum_{j=1}^M \x_j \left( \int_0^{\node_m} \ell_j(s)~ds \right) =
    \int_0^{\node_m} \sum_{j=1}^M \x_j \ell_j(s)~ds =
    \int_0^{\node_m} x(s)~ds
\end{equation}
for $m=1, \ldots, M$ so that
\begin{equation}
    \Q \x = \begin{bmatrix} \int_0^{\node_1} x(s)~ds \\ \vdots \\ \int_0^{\node_M} x(s)~ds \end{bmatrix}.
\end{equation}
Hence, multiplying a vector $\x \in \R^M$ by $\Q$ generates a vector that has the associated polynomial integrated from zero to the quadrature nodes $\node_m$ as components.
Applying $\Phi$ to $\Q \x$ fits a polynomial through the points $(\node_m, \int_0^{\tau_m} x(t)~dt)$ so that\footnote{Strictly speaking we should write $(Qx)(t)$ since $Qx \in P^{M-1}$ but we omit the argument $t$ for less cluttered notation.}
\begin{equation}
    \label{eq:Qnonbold}
    Q x := \sum_{j=1}^M \left( \int_0^{\node_j} x(s)~ds \right) l_j(t).
\end{equation}
The mappings $\Q$, $Q$ and $\Phi$ commute, see Figure~\ref{fig:mapping}.
\begin{figure}[t]
    \centering\label{fig:mapping}
    \includegraphics[width=0.28\linewidth]{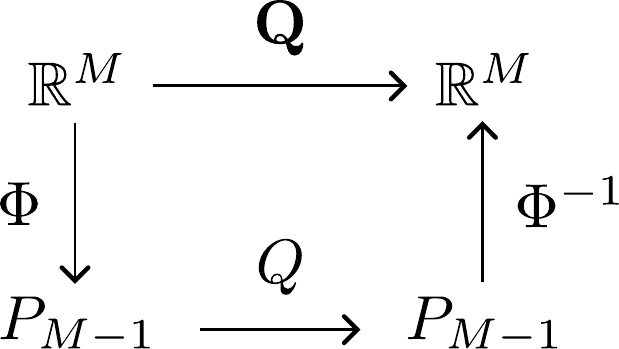}
    \caption{Bijective mapping between $\R^M$ and $P_{M-1}$.}
\end{figure}
\begin{prop}
    \label{prop:monomials}
For $1 \leq n \leq M-1$, let $\vect{\tau}^{n} \in \R^M$ be the vector $\left( \node_1^n, \ldots, \node_M^n\right)$  and $\tau^n \in P_{M-1}$ the monomial $t \mapsto t^n$.
Then
\begin{equation}
    \vect{\tau}^n \cong \tau^n.
\end{equation}
\end{prop}
\begin{proof}
By definition of the mapping, $\Phi \vect{\tau}^n$ is the polynomial of degree $M-1$ interpolating the points $(\node_j, \node_j^n)$ for $j=1, \ldots, M$.
Since the monomial $\tau^n \in P_{M-1}$ interpolates these points and the interpolating polynomial is unique, we must have $\Phi \vect{\tau}^n = \tau^n$.
\end{proof}
\begin{remark}
Note that we can still apply $\Phi$ to vectors $\vect{\tau}^n \in \R^M$ for $n > M-1$, but we will no longer obtain the monomial $\tau^n$. Instead, we obtain the polynomial
\begin{equation}
    p(t) = \sum_{j=1}^M \node_j^n l_j(t)
\end{equation}
of degree $M-1$ that interpolates the points $(\node_j, \node_j^n)$.
\end{remark}
\begin{prop}
    \label{prop:Q}
With the definitions from Proposition~\ref{prop:monomials} we have
\begin{equation}
    \Q \vect{\tau}^n \cong \frac{\node^{n+1}}{n+1} %= \frac{1}{n+1} \vect{\tau}^{n+1} \cong \frac{1}{n+1}
\end{equation}
for $1 \leq n \leq M-2$.
\end{prop}
\begin{proof}
Using~\eqref{eq:Qnonbold} we have
\begin{equation}\label{eq:q_on_mon}
     Q \Phi \vect{\tau}^n = \sum_{j=1}^{M} \left( \int_0^{\node_j} s^n~ds \right) \ell_j(t) = \frac{1}{n+1} \sum_{j=1}^M \tau_j^{n+1} \ell_j(t).
\end{equation}
Because $l_j(\node_m) = 1$ if $j=m$ and zero otherwise, evaluating this polynomial at the nodes $\node_j$ when applying $\Phi^{-1}$ recovers the values $\tau_j^{n+1}/(n+1)$ so that
\begin{equation}
    \Phi^{-1} Q \Phi \vect{\tau}^n = \frac{1}{n+1} \vect{\tau}^{n+1}.
\end{equation}
Because $\Phi$ is an isomorphism, we can apply it to both sides of the equation and, since $n+1 \leq M-1$, use Proposition~\ref{prop:monomials} to get
\begin{equation}
    \label{eq:Qphitau_n}
    Q \Phi \vect{\tau}^n = \frac{1}{n+1} \Phi(\vect{\tau}^{n+1}) = \frac{1}{n+1} \tau^{n+1}.
\end{equation}
Noting that $\Q \vect{\tau}^n \cong Q \Phi \vect{\tau}^n$ completes the proof.
\end{proof}
\begin{prop}
    \label{prop:Qdm}
Consider a set of nodes $0 \leq \node_1 < \dots < \node_M \leq 1$ with $\node_1 > 0$ and $m \in \mathbb{N}$.
Let
\begin{equation}\label{eq:Qdm}
    \Qdm{m} := \operatorname{diag}
        \left(\frac{\node_1}{m}, \dots, \frac{\node_M}{m}\right)
\end{equation}
be a diagonal matrix with entries $\node_j/m$ and $\qdm{m} := \Phi \Qdm{m} \Phi^{-1}$.
For $1\leq n \leq M-2$, it holds that
\begin{equation} \label{eq:qdm_on_mon}
    \Qdm{m} \vect{\tau}^n \cong \frac{\node^{n+1}}{m}
\end{equation}
where $\tau^{n+1}$ is again the monomial $t \mapsto t^{n+1}$.
\end{prop}
\begin{proof}
We have
\begin{equation}
    \Qdm{m} \vect{\tau}^{n} = \begin{bmatrix} \frac{\node_1}{m} \node_1^n \\ \vdots \\ \frac{\node_M}{m} \node_m^n \end{bmatrix} = \frac{1}{m} \begin{bmatrix} \node_1^{n+1} \\ \vdots \\ \node_M^{n+1} \end{bmatrix} = \frac{1}{m} \vect{\tau}^{n+1}.
\end{equation}
Applying $\Phi$ and using Proposition~\ref{prop:monomials} yields
\begin{equation}
    \Qdm{m} \vect{\tau}^n \cong \Phi \Qdm{m} \vect{\tau}^n = \Phi \frac{\vect{\tau}^{n+1}}{m} = \frac{\tau^{n+1}}{m}.
\end{equation}
\end{proof}
\begin{prop}
    \label{prop:Qdminv}
For nodes $0 < \node_1 < \dots < \node_M \leq 1$, we have
\begin{equation}
    \label{eq:qdmInv_on_mon}
    \Qdm{m}^{-1}\vect{\tau}^{n+1} = m \vect{\tau}^n \cong m \tau^n = \qdm{m}^{-1} \node^{n+1}
\end{equation}
for $1 \leq n \leq M-1$.
\end{prop}
\begin{proof}
Because of $\node_1 > 0$, the matrix $\Qdm{m}$ is invertible with inverse
\begin{equation}
    \Qdm{m}^{-1} = \text{diag}\left( \frac{m}{\node_1}, \ldots, \frac{m}{\node_M} \right)
\end{equation}
and thus
\begin{equation}
    \Qdm{m}^{-1} \vect{\tau}^{n+1} = \text{diag}\left( \frac{m}{\node_1} \node_1^{n+1}, \ldots, \frac{m}{\node_M} \node_M^{n+1} \right) = m \vect{\tau}^n.
\end{equation}
For $1 \leq n \leq M-1$ we have $\vect{\tau}^n \cong \tau^n$ by Proposition~\ref{prop:monomials}.
Finally,
\begin{equation}
    \qdm{m}^{-1} \tau^{n+1} = \Phi \Qdm{m}^{-1} \Phi^{-1} \tau^{n+1} = \Phi \Qdm{m}^{-1} \begin{bmatrix} \node_1^{n+1} \\ \vdots \\ \node_m^{n+1} \end{bmatrix} = m \Phi \begin{bmatrix} \node_1^n \\ \vdots \\ \node_M^n \end{bmatrix} = m\tau^n
\end{equation}
using Proposition~\ref{prop:monomials}.
\end{proof}

These results will be used in the proofs in the next sections.
\subsubsection{\MINSRNS preconditioning}
\label{sec:nonstiff}
Here we introduce the \MINSRNS SDC preconditioner $\QDelta=\Qdm{M}$ with constant coefficients that is suited for non-stiff problems.

\begin{theorem}\label{th:Q-Qd}
    For any set of collocation nodes $0 < \node_1 < \dots < \node_M \leq 1$, the matrix $\Q - \Qdm{M}$ is nilpotent with index $M$. For the \MINSRNS preconditioner, setting $\QDelta=\Qdm{M}$, it holds that $\rho(\KNS)=0$.
\end{theorem}
\begin{proof}
Let $\vect{\tau}^{M-1} \in \R^M$.
Then it holds that
\begin{equation}
    \Q \vect{\tau}^{M-1} = \begin{bmatrix} \int_0^{\node_1} t^{M-1}~ds \\ \vdots \\ \int_0^{\node_M} t^{M-1}~ds\end{bmatrix} = \frac{1}{M} \vect{\tau}^M
\end{equation}
and
\begin{equation}
    \Qdm{M} \vect{\tau}^{M-1} = \frac{1}{M} \vect{\tau}^{M}
\end{equation}
so that $\vect{\tau}^{M-1} \in \ker(\Q - \Qdm{M})$.
By Proposition~\ref{prop:monomials} and because $\Phi$ is an isomorphism, $\tau^{M-1} = \Phi \vect{\tau}^{M-1}$ is in the kernel of $Q - \qdm{m}$.
For $1 \leq n \leq M-2$ the Propositions~\ref{prop:Q} and~\ref{prop:Qdm} yield
\begin{equation}\label{eq:tM-1}
    (\Q-\Qdm{M})\vect{\node}^n
    \cong
    \left(\frac{1}{n+1}-\frac{1}{M}\right) \node^{n+1}
\end{equation}
Now consider any polynomial $p \in P_{M-1}$ with $p(t) = \sum_{j=1}^M p_j t^j$.
Then,
\begin{equation}
    (Q-\qdm{M}) p =  \sum_{j=1}^{M-1} p_j \left( \frac{1}{j+1} - \frac{1}{M} \right) t^{j+1} = \sum_{j=2}^M p_{j-1} \left( \frac{1}{j} - \frac{1}{M} \right) t^j
\end{equation}
so that $(Q - \qdm{M})p \in \left( P_{M-1} \setminus P_0 \right) \cup \{0\}$.
By induction, we get
\begin{equation}
    (Q-\qdm{M})^k p
    \in \left( P_{M-1} \setminus P_{k-1} \right) \cup \{0\}.
\end{equation}
Applying $(Q-\qdm{M})$ $M$ times yields
\begin{equation}\notag
    (Q-\qdm{M})^M p \in \left( P_{M-1} \setminus P_{M-1} \right) \cup \{0\} = \{0\}.
\end{equation}
As $p\in P_{M-1}$ was arbitrary, we have $(Q-\qdm{M})^M = 0$.
Because $\Phi$ is an isomorphism we find that $(\Q - \Qdm{M})^M = 0$ and therefore $\rho(\KNS)=0$.
\end{proof}

\begin{remark}\label{rem:errorTerms}
    In the proof of Theorem~\ref{th:Q-Qd}, one can interpret
    $p(\tau)$ as a polynomial representation of the
    collocation error, since the iteration matrix is approximated
    by $\KNS$.
    Hence, each SDC iteration using $\Qdm{M}$ preconditioning
    improves the solution quality by
    removing the lowest order term in the error,
    up to the point where there is no term left.
\end{remark}

Note that the \MINSRNS coefficients are different from the ones derived by van der Houwen and Sommeijer~\cite[Sec.~3.3.1]{houwen1991iterated}.
They suggest to use a diagonal matrix that satisfies
\begin{equation}\notag
    \QDelta^{-1}\vect{\tau} = \Q^{-1}\vect{\tau},
\end{equation}
which leads to $\QDelta=\Qdm{1}$ and corresponds to using an Implicit
Euler step between $t_0$ and the nodes time $t_m$ (\IEpar).
The reason is that they aimed to improve convergence of the non-stiff components of the solution for large time-steps, which is different from minimizing the spectral radius of $\KNS$.
%Numerical experiments show that the \MINSRNS preconditioner from Theorem~\ref{th:Q-Qd} that analytically minimizes $\rho(\KNS)$ does yield better numerical results than the one proposed in \cite{houwen1991iterated}.

\subsubsection{\MINSRS preconditioning}
\label{sec:det}
Here we introduce a SDC preconditioner with constant coefficients that is suited
for stiff problems.

\begin{definition}\label{def:MIN-SR-S}
    Consider a set of collocation nodes $0 < \node_1 < \dots < \node_M \leq 1$.
    We call a diagonal matrix $\QDelta$ with increasing diagonal entries that minimizes
    \begin{equation}\label{eq:MIN-SR-S}
          \left|\det\left[(1-t)\Imat + t\QDelta^{-1}\Qmat\right] - 1\right|,
          \quad \forall t \in \{\node_1,\dots,\node_{M}\}
    \end{equation}
    a \MINSRS preconditioner for SDC.
    Such a preconditioner finds a local minimum for $\rho(\KS)$.
\end{definition}

As mentioned above, using the spectral radius as objective function makes the optimization problem very challenging to solve numerically.
Instead, we search for diagonal coefficients such that $\KS$ is nilpotent.
While we have no guarantee that such coefficients exist for every $M$, it is known that for $M=2$ there are two possible minimizers but only one where the coefficients are ordered~\cite{houwen1991iterated}.
If such coefficients exist for any $M$, the following holds
\begin{equation}\label{eq:detNilpotent}
    \forall t \in \R,\quad \det\left[\Imat + t(\QDelta^{-1}\Qmat - \Imat)\right] - 1 = 0.
\end{equation}
Since $\det\left[\Imat + t(\QDelta^{-1}\Qmat - \Imat)\right] - 1$ is a polynomial in $t$ of degree $M$, we only need to check~\eqref{eq:detNilpotent} for $M+1$ points.
Because the equation is trivially satisfied for $t=0$, checking it for nodes $0 < \node_1 < \dots < \node_{M}$ is sufficient to show~\eqref{eq:MIN-SR-S}.

Note that currently there is no theory showing whether~\eqref{eq:MIN-SR-S} has one, several or no solution.
We use \texttt{MINPACK}'s \texttt{hybrd} algorithm implemented in \texttt{scipy} to find diagonal coefficients that minimize locally the spectral radius of $\KS$. 
For $M=4$ for example, the \MINSRS coefficients for \RadauRight nodes shown in Appendix~\ref{ap:coeffs} give $\rho(\KS)=0.00024$.
However, this approach does not ensure that the diagonal coefficients are increasingly ordered.
Because order coefficients led to better stability in our numerical experiments, see the discussion in §\ref{sec:numStability}, we used a particular choice of starting value for the minimization .
Because the \MINSRNS coefficients and the increasingly ordered coefficients minimizing the stiff spectral radius $\rho(\KS)$ are similar, we used \MINSRNS as starting values for the optimization finding the \MINSRS coefficients.
This yielded increasingly ordered coefficients up to $M=4$.
For larger values of $M$,
we observed that we can fit a power-law of the form $\alpha t^{\beta}$ on $[0, 1]$
through the points $(\tau_i^M, Md_i^M)$, $i=1, \ldots, M$, where $\vect{d}^M$ are the 
increasingly ordered coefficients that minimize $\rho(\KS)$, and evaluate it to produce a starting value for $M+1$. 
Hence, we propose the following incremental procedure to provide good starting values for the optimization to compute $\vect{d}^{M+1}$.
Assuming that $\vect{d}^{M}$ is known:
\begin{enumerate}
    \item find values for $\alpha$ and $\beta$ such that the power-law minimizes the distance to the points 
        $(\tau_i^M, Md_i^M)$ in the $L_2$ norm,
    \item compute $\widetilde{\vect{d}}^{M+1} = \alpha t^\beta / (M+1)$
    for $t \in \vect{\node}^{M+1}$,
    \item find a numerical solution for \eqref{eq:MIN-SR-S} using
    $\widetilde{\vect{d}}^{M+1}$ as initial guess.
\end{enumerate}
Iterating this process up to a desired $M$ yielded increasingly ordered coefficients with a very
small $\rho(\KS)$ in all numerical experiments.

\begin{remark}\label{rem:zeroNode}
The assumption $\node_1 \neq 0$ in Definition~\ref{def:MIN-SR-S} guarantees that $\QDelta$ is not singular.
If the first collocation node is zero, \eg for \Lobatto nodes, the collocation matrix takes the form
\begin{equation}\notag
    \Q =
    \begin{bmatrix}
        x & \vect{y}^\top\\
        \vect{q} & \widetilde{\Q}
    \end{bmatrix},
\end{equation}
where $\vect{y}, \vect{q} = [\vect{q}_1,  \dots, \vect{q}_{M-1}]^\top \in \R^{M-1}$,
and $x \in \R$.
Then,~\eqref{eq:linear_collocation_problem} can be rewritten as
\begin{equation}\notag
    \begin{bmatrix}
        u_2 \\
        \vdots \\
        u_M
    \end{bmatrix}
    - \widetilde{\Q}
    \begin{bmatrix}
        f(u_2) \\
        \vdots \\
        f(u_M)
    \end{bmatrix} =
    \begin{bmatrix}
        u_0 + \vect{q}_1f(u_0) \\
        \vdots \\
        u_0 + \vect{q}_{M-1}f(u_0)
    \end{bmatrix}
\end{equation}
since $u(\node_1) = u_0$.
The matrix $\widetilde{\Q}$ is still a collocation matrix, but now based on the nodes $0 < \node_2 < \dots < \node_M$.
We can apply the approach described in Definition~\ref{def:MIN-SR-S} for $\widetilde{\Q}$ to determine a diagonal $\widetilde{\Q}_\Delta$ and add a zero coefficient to build the diagonal $\QDelta$ preconditioner for the original node distribution.
\end{remark}

\subsubsection{\MINSRFLEX preconditioning}
%\subsubsection{Variable preconditioning in the stiff case}
\label{sec:dnodes}
The discussion in §\ref{sec:det} illustrates the difficulty of finding a single set of diagonal coefficients that minimize $\rho(\KS)$ analytically.
Therefore, here we consider a series of preconditioners that change from one iteration to the next, as was already suggested by Weiser~\cite[Sec.~4.2]{weiser2015faster}.
Let $\Qd^{(k)}$ be the preconditioner used in the $k^{\rm{th}}$ iteration of SDC.
Telescoping the error iteration~\eqref{eq:error_iterations} gives
\begin{equation}\label{eq:error_iterations_var}
\e^{k} = \K^{(k)}(z) \dots \K^{(1)}(z) \e^{0}, \quad
\K^{(k)}(z) = z \left(\I - z \Qd^{(k)}\right)^{-1}\left(\Q - \Qd^{(k)}\right).
\end{equation}
Using the same calculation as for the iteration matrix in the stiff limit $\KS$,
we get
\begin{equation}\notag
\lim_{|z|\rightarrow \infty}\e^{k} = \KS^{(k)}\dots \KS^{(1)}\e^{0}, \quad
\KS^{(k)} = \lim_{|z|\rightarrow \infty}\K^{(k)}(z) =
\I - \left(\Qd^{(k)}\right)^{-1}\Q.
\end{equation}
We use this result to introduce the following $k$-dependent preconditioning.
\begin{theorem}\label{th:QdiQ-I}
    For any set of nodes with $\node_1 > 0$, we have
    \begin{equation}\notag
        \left(\I - \Q_{\Delta, M}^{-1} \Q \right)
        \dots
        \left(\I - \Q_{\Delta, 2}^{-1} \Q \right)
        \left(\I - \Q_{\Delta, 1}^{-1} \Q \right) = \vect{0},
    \end{equation}
    where $\Qdm{m}$ is defined in \eqref{eq:Qdm}.
    Hence, using successive diagonal preconditioning
    $\QDelta^{(k)}=\Qdm{k}$ with
    $k\in\{1,\dots,M\}$ provides SDC iterations
    such that $\lim_{|z|\rightarrow \infty}\e^{k}=0$,
    with $\e^{k}$ being the iteration error defined in
    \eqref{eq:error_iterations_var}.
    We call this preconditioning \MINSRFLEX.
\end{theorem}
\begin{proof}
Since $\Q \cong Q$, $\Qdm{m}^{-1} \cong \qdm{m}^{-1}$ and $\I \cong I$, we have
\begin{equation}
    \label{eq:qdqiso}
 \left(\I - \Q_{\Delta, m}^{-1} \Q \right) \cong \left(I - \qdm{m}^{-1}Q\right).
\end{equation}
For any monomial $\tau^n \in P_{M-1}$, Proposition~\ref{prop:monomials} yields
\begin{align}
    \left( I - \qdm{m}^{-1} Q \right) \tau^n &= \tau^n - \qdm{m}^{-1} Q \Phi \vect{\tau}^n.
\end{align}
With the help of equation~\eqref{eq:Qphitau_n} we find that
\begin{equation}
    \left( I - \qdm{m}^{-1} Q \right) \tau^n = \tau^n - \qdm{m}^{-1} \frac{\tau^{n+1}}{n+1}
\end{equation}
Finally, Proposition~\ref{prop:Qdminv} gives us
\begin{equation}
     \left( I - \qdm{m}^{-1} Q \right) \tau^n = \tau^n - \frac{m \tau^n}{n+1} = \left(1 - \frac{m}{n+1} \right) \tau^n
\end{equation}
If $m = n+1$, the right hand side is zero and therefore $\tau^{n} \in \ker(I - \qdm{n+1}^{-1})$.

Let $p \in P_{M-1}$, $p(t) = \sum_{j=0}^{M-1} p_j t^j$ be an arbitrary polynomial of degree $M-1$.
Then
\begin{equation}
    \left( I - \qdm{1}^{-1} Q \right) p(t) = \frac{1}{2} p_1 t + \frac{2}{3} p_2 t^2 + \ldots + \frac{M-1}{M} p_{M-1}t^{M-1} \in \left( P_{M-1} \setminus P_0 \right) \cup \{ 0 \}
\end{equation}
Applying $I - \qdm{2}^{-1}Q$ removes the linear term so that
\begin{equation}
    \left( I - \qdm{2}^{-1} Q \right) \left( I - \qdm{1}^{-1} Q \right) p \in \left( P_{M-1} \setminus P_1 \right) \cup \{ 0 \}.
\end{equation}
Repeating this until $k = M$ yields
\begin{equation}\label{eq:dnodek}
    \left(I-\qdm{M-2}^{-1}Q\right)
    \circ\dots\circ
    \left(I-\qdm{2}^{-1}Q\right)
    \circ
    \left(I-\qdm{1}^{-1}Q\right)(p)
    \in P_{M-1} \setminus P_{M-1} \cup \{0 \},
\end{equation}
that is, $p$ gets mapped to the zero polynomial. 
Since $p$ was arbitrary, the mapping must be the zero mapping.
Using~\eqref{eq:qdqiso} shows that
\begin{equation}\notag
\left(\I - \Q_{\Delta, M}^{-1} \Q \right)
\dots
\left(\I - \Q_{\Delta, 2}^{-1} \Q \right)
\left(\I - \Q_{\Delta, 1}^{-1} \Q \right) = \vect{0}.
\end{equation}
\end{proof}

Note that a similar observation to what was stated in Remark~\ref{rem:errorTerms} concerning
the error reduction can be made for the \MINSRFLEX preconditioner from Theorem~\ref{th:QdiQ-I}.
As before, Remark~\ref{rem:zeroNode} holds \MINSRFLEX preconditioners for any node distribution with
$\node_1=0$.
Theorem~\ref{th:QdiQ-I} defines the preconditioner for a fixed number of iterations $k \leq M$ and lets the user choose the coefficients for $k>M$.
Since this preconditioning is tailored to stiff problems, we use the \MINSRS preconditioning from Definition~\ref{def:MIN-SR-S} in combination with \MINSRFLEX when $k>M$ in all numerical experiments.

\section{Convergence order and stability}
\label{sec:convStability}
This section investigates the convergence order and stability of parallel SDC with a fixed number of sweeps.
This approach is easier to compare against classical RKM than using SDC with a residuum-based stopping criterion~\cite{speck2018parallelizing}.
We use nodes from a Legendre distribution but the reader can generate plots for other choices with the provided code~\cite{speck2024pySDC}.
In all cases, the initial solution $\vect{u}^0$ for the SDC iteration is generated by copying the initial value $u_0$ to all nodes.

\subsection{Convergence order}
As a rule of thumb, SDC's order of converge increases by at least one per iteration.
This has been proved for implicit and explicit Euler methods as sweepers~\cite[Thm.~4.1]{dutt2000spectral} as well as for higher order correctors \cite[Thm.~3.8]{christlieb2009comments}.
However, except for equidistant nodes, sweepers of higher order are not guaranteed to provide more than one order per sweep.
For \LU-SDC there is numerical evidence that it increases order by one per sweep~\cite{weiser2015faster} but no proof seems to exist.
Van der Houwen et al.~\cite[Thm.~2.1]{houwen1992embedded} show that for any diagonal preconditioner $\QDelta$ the order of the method after $K$ sweeps is $\min(K, p^*)$, where $p^*$ is the order of the underlying collocation method.
We confirm this numerically for the Dahlquist test problem~\eqref{eq:dahlquist} with $\lambda = i$ and $T=2\pi$.

\def\captionText#1{
    Convergence of SDC for the Dahlquist test equation with #1 preconditioner
    using $K=1,\ldots,M$ sweeps per per time step.
    Dashed lines with slopes one to $K$ are shown as a guide to the eye.
    Left: $M=4$ \RadauRight nodes, right: $M=5$ \Lobatto nodes.
}
\begin{figure}[!htb]
    \label{fig:conv_MIN-SR-NS}
    \includegraphics[width=0.48\linewidth]{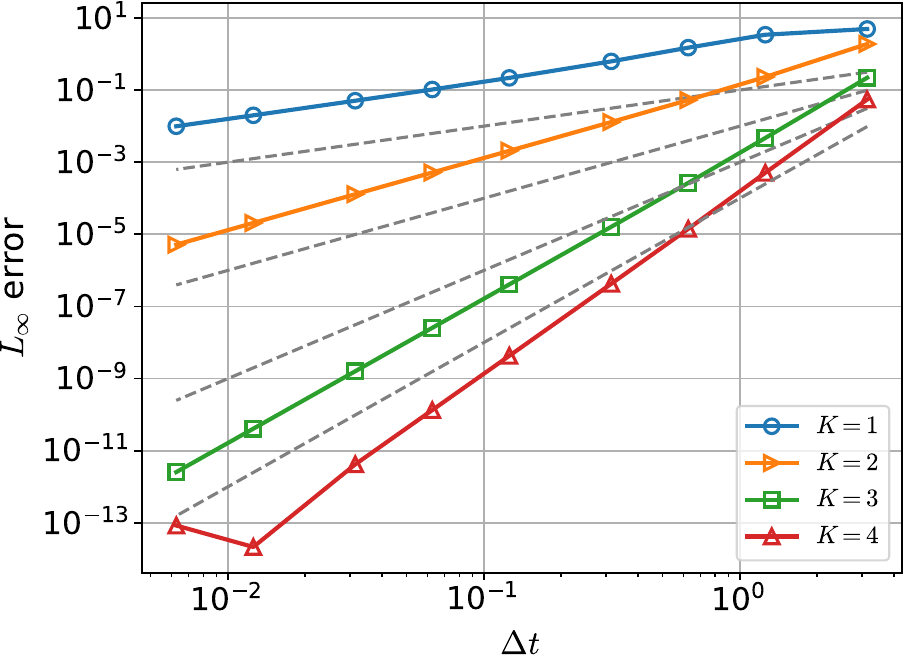}%
    \hfill%
    \includegraphics[width=0.48\linewidth]{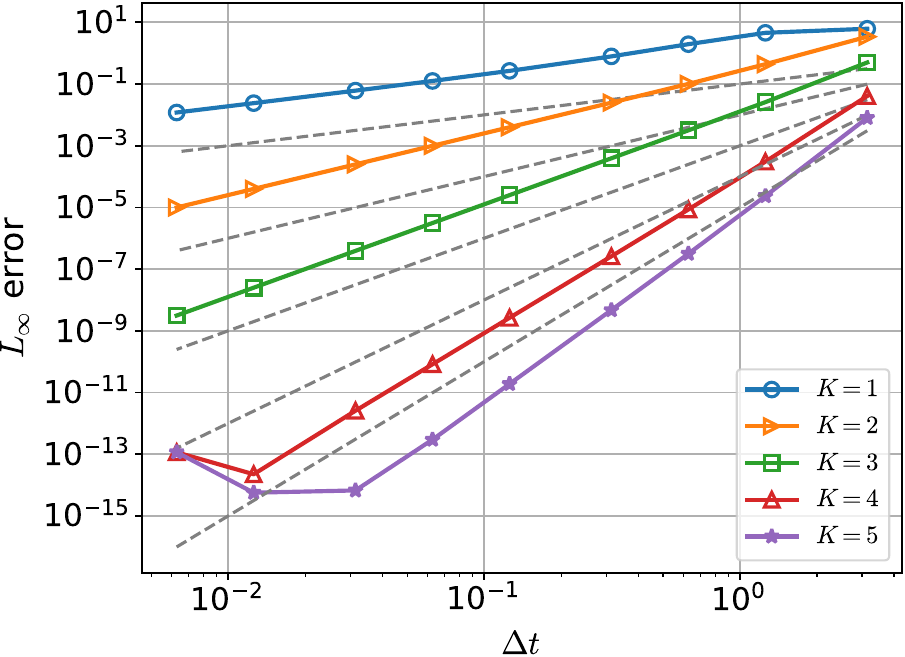}%
    \caption{\captionText{\MINSRNS}}
\end{figure}
Figure~\ref{fig:conv_MIN-SR-NS} shows the $L_{\infty}$-error against the analytical solution for SDC with the \MINSRNS preconditioner with $M=4$ \RadauRight nodes (left) and $M=5$ \Lobatto nodes (right)
for $K=1,2,\dots,M$ sweeps.
The underlying collocation methods are of order 7 and 8 so that the order of SDC is determined by $K$.
For \RadauRight nodes, SDC gains one order per sweep for $K=1$ and $K=2$ while the third sweep increases the order by two.
The same happens for \Lobatto nodes when going from $K=3$ to $K=4$ sweeps.
This unexpected order gain has also been observed for other configurations of the
\MINSRNS preconditioner but we do not yet have a theoretical explanation.
Increased order of \MINSRNS is also observed for more complex problems, see \S\ref{sec:efficiencyLorenz}.

\begin{figure}
    \label{fig:conv_MIN-SR-S}
    \includegraphics[width=0.48\linewidth]{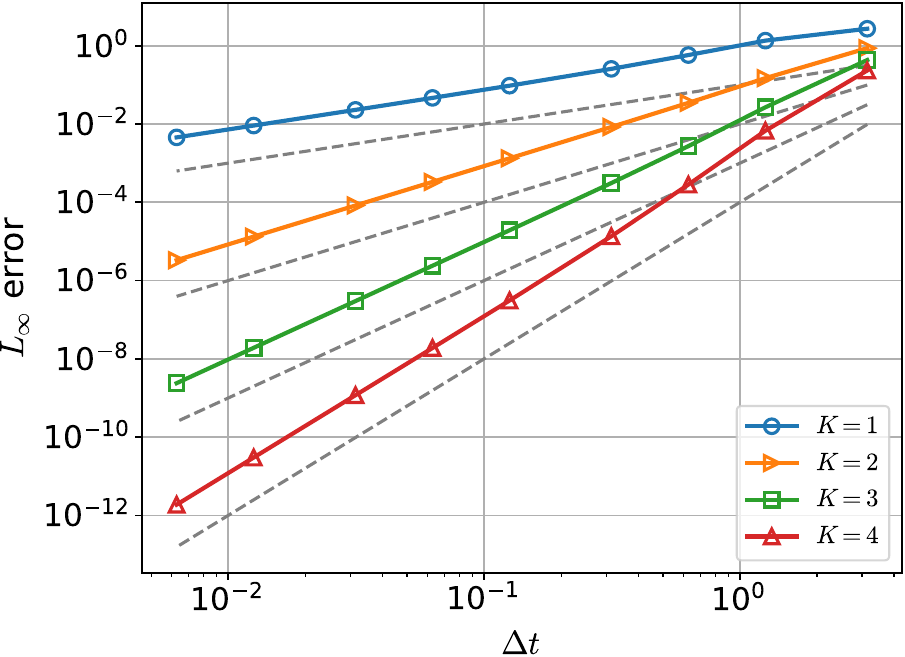}%
    \hfill%
    \includegraphics[width=0.48\linewidth]{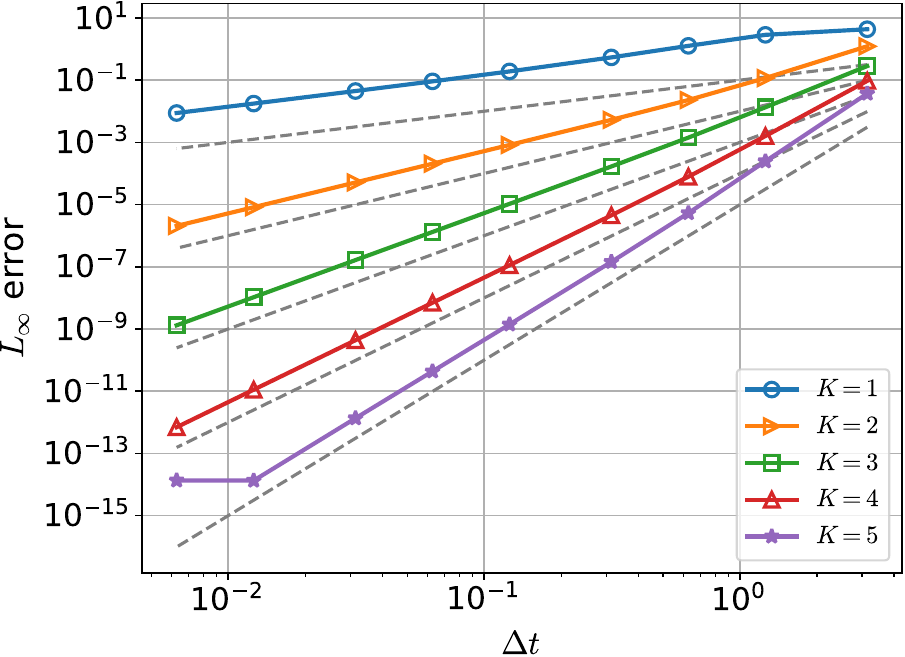}%
    \caption{\captionText{\MINSRS}}
\end{figure}

\begin{figure}
    \label{fig:conv_MIN-SR-FLEX}
    \includegraphics[width=0.48\linewidth]{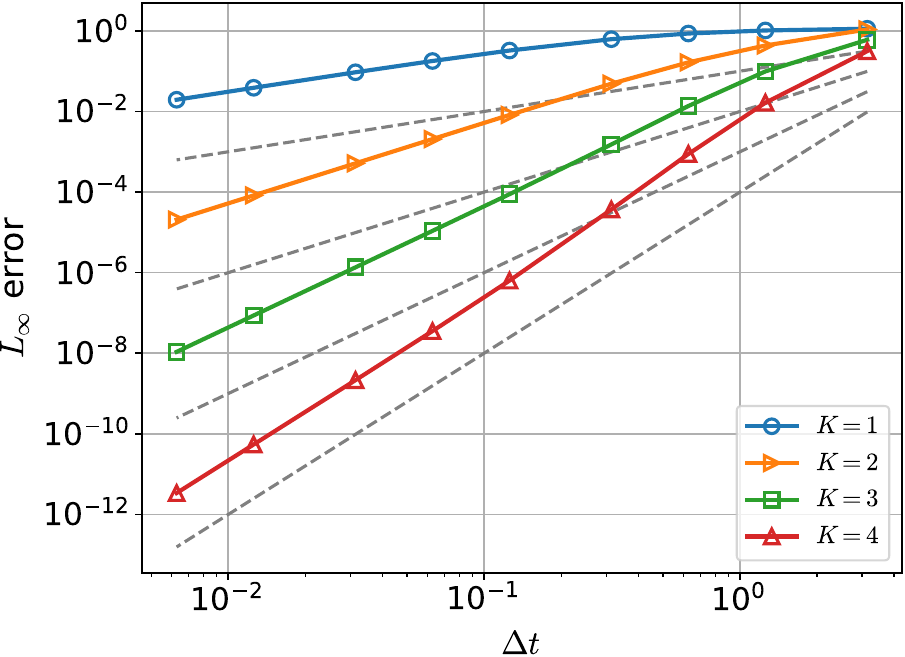}%
    \hfill%
    \includegraphics[width=0.48\linewidth]{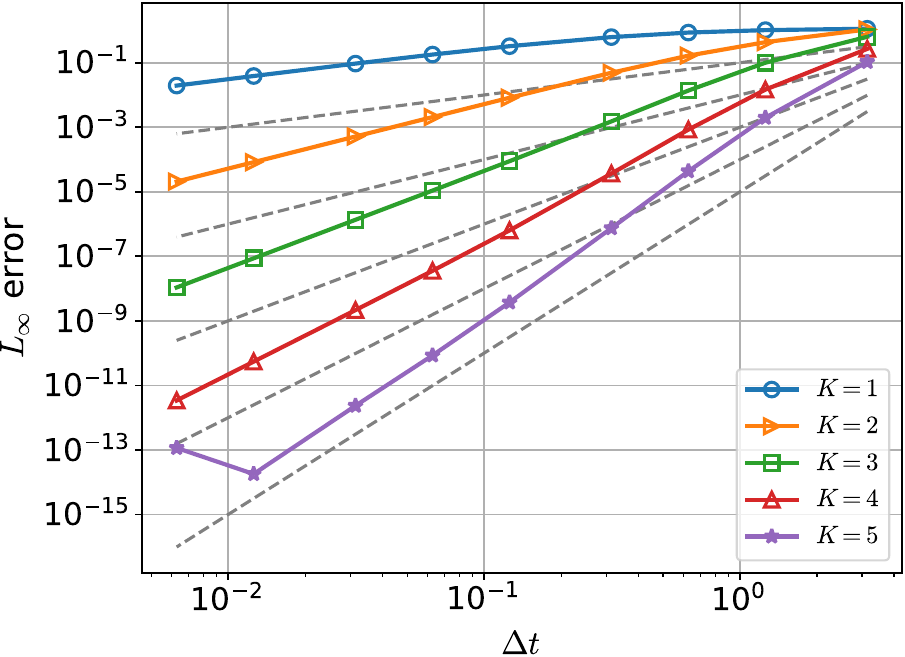}%
    \caption{\captionText{\MINSRFLEX}}
\end{figure}

Figure~\ref{fig:conv_MIN-SR-S} shows convergence of SDC with \MINSRS preconditioner while Figure~\ref{fig:conv_MIN-SR-FLEX} shows convergence for \MINSRFLEX.
In both cases the order increases by one per sweep but without the additional gains we oberserved for \MINSRNS.
Also, errors for \MINSRS and \MINSRFLEX are generally higher than for \MINSRNS.
This is expected as \MINSRNS is optimized for non-stiff problems and the Dahlquist problem with $\lambda = i$ is not stiff.
Note that no results seem to exist in the literature that analyze convergence order of a nonstationary SDC iteration like \MINSRFLEX where the preconditioner changes in every iteration.

\subsection{Numerical stability}
\label{sec:numStability}

We investigate stability of parallel SDC by numerically computing the border of its stability region
\begin{equation}
    \mathcal{S}_C = \left\{z \in \C \;\text{s.t}\; |R(z)| = 1 \right\},
\end{equation}
where $R$ is the stability function of a given SDC configuration.
A method is called A-stable if $\mathcal{S}_C \subset \mathbb{C}^{-}$, that is, if the stability domain includes the negative complex half-plane.
We use $M=4$ \RadauRight nodes from a Legendre distribution for all experiments. 
Other configurations can again be analyzed using the provided code.

\def\captionText#1{
    Stability region for SDC with $M=4$ \RadauRight nodes using the #1 preconditioner for $K=1,2,3,4$.
    The gray zones are unstable areas in the complex plane.
}
\begin{figure}
    \label{fig:stab_PIC}
    \includegraphics[width=0.24\linewidth]{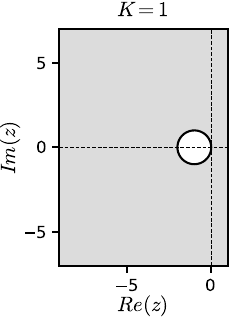}%
    \hfill%
    \includegraphics[width=0.24\linewidth]{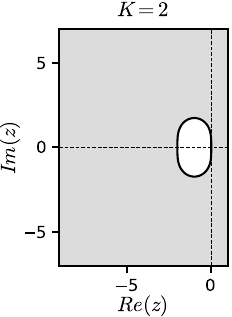}%
    \hfill%
    \includegraphics[width=0.24\linewidth]{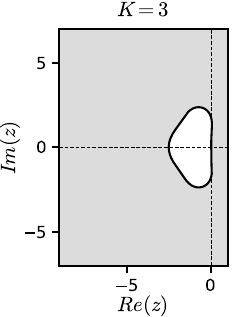}%
    \hfill%
    \includegraphics[width=0.24\linewidth]{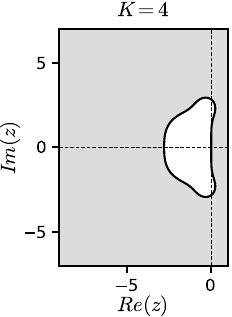}%
    \caption{\captionText{\PIC}}
\end{figure}
Figure~\ref{fig:stab_PIC} shows stability for SDC with \PIC preconditioner
(Picard iteration).
This is a fully explicit method and, as expected, not $A$-stable for any $K$.
Interestingly, $K$-many sweeps reproduce the stability contour of the explicit RKM of order $K$ with $K$ stages.
In particular, we recognize the stability regions of explicit Euler for $K=1$ and of the classical explicit RKM of order $4$ for $K=4$.

\begin{figure}
    \label{fig:stab_MIN-SR-NS}
    \includegraphics[width=0.24\linewidth]{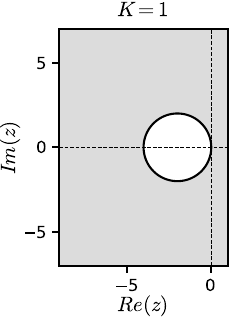}%
    \hfill%
    \includegraphics[width=0.24\linewidth]{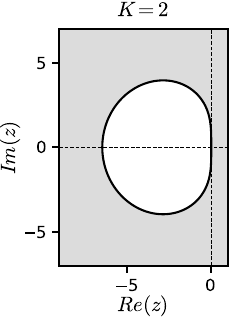}%
    \hfill%
    \includegraphics[width=0.24\linewidth]{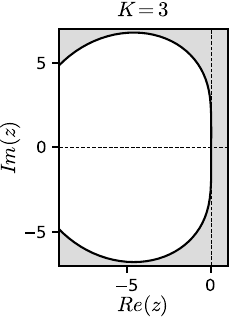}%
    \hfill%
    \includegraphics[width=0.24\linewidth]{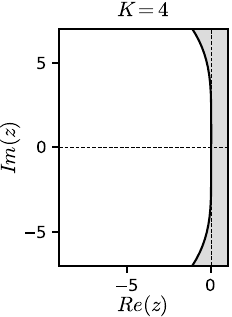}%
    \caption{\captionText{\MINSRNS}}
\end{figure}
Figure~\ref{fig:stab_MIN-SR-NS} shows stability regions for \MINSRNS.
As for the Picard iteration, the method is not $A$-stable for any number of sweeps 
but the stability regions are significantly larger.
That \MINSRNS is not $A$-stable is not unexpected since the coefficients are optimized
for the non-stiff case where $|z| \simeq 1$.

\begin{figure}
    \label{fig:stab_MIN-SR-S}
    \includegraphics[width=0.24\linewidth]{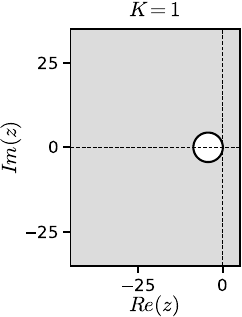}%
    \hfill%
    \includegraphics[width=0.24\linewidth]{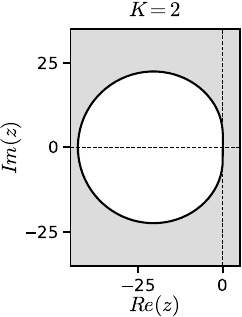}%
    \hfill%
    \includegraphics[width=0.24\linewidth]{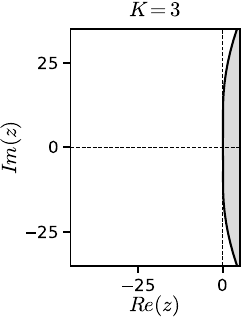}%
    \hfill%
    \includegraphics[width=0.24\linewidth]{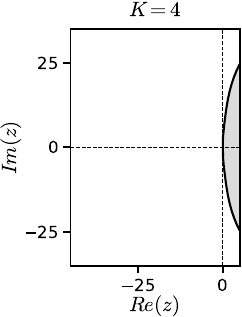}%
    \caption{\captionText{\MINSRS}Note that $\lambda$ ranges are five times larger than
        for the previous Figures~\ref{fig:stab_PIC} and \ref{fig:stab_MIN-SR-NS}.}
\end{figure}

\begin{figure}
    \label{fig:stab_MIN-SR-FLEX}
    \includegraphics[width=0.24\linewidth]{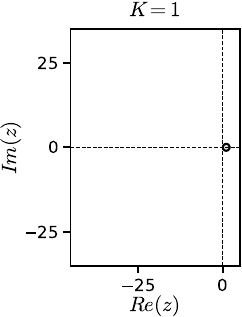}%
    \hfill%
    \includegraphics[width=0.24\linewidth]{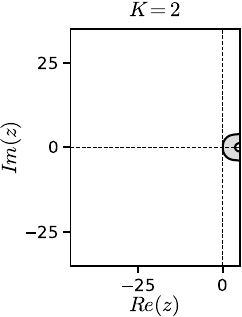}%
    \hfill%
    \includegraphics[width=0.24\linewidth]{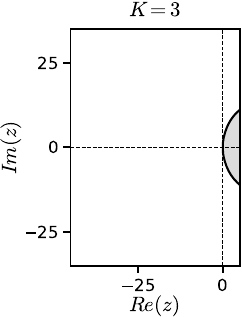}%
    \hfill%
    \includegraphics[width=0.24\linewidth]{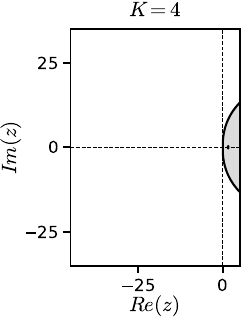}%
    \caption{\captionText{\MINSRFLEX}}
\end{figure}
Figures~\ref{fig:stab_MIN-SR-S} and \ref{fig:stab_MIN-SR-FLEX} show stability regions for \MINSRS and \MINSRFLEX.
While SDC with \MINSRS appears to be $A$-stable for $K \geq 3$,
SDC with \MINSRFLEX seems $A$-stable for any number of sweeps $K$.
However, this only holds for \RadauRight nodes.
For example, SDC with $M=5$ \Lobatto nodes and $K=4$ sweeps is not $A$-stable (stability regions not shown here).
A theoretical investigation is left for later work.
\begin{figure}
    \label{fig:stab_LU}
    \includegraphics[width=0.24\linewidth]{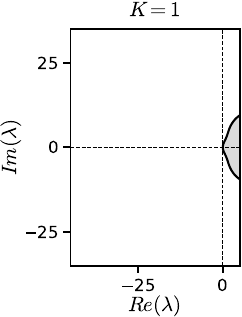}%
    \hfill%
    \includegraphics[width=0.24\linewidth]{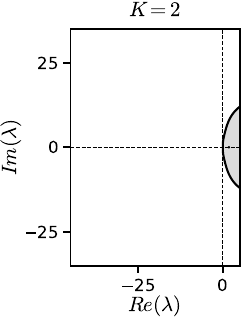}%
    \hfill%
    \includegraphics[width=0.24\linewidth]{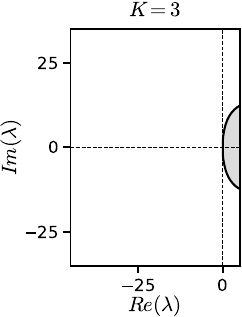}%
    \hfill%
    \includegraphics[width=0.24\linewidth]{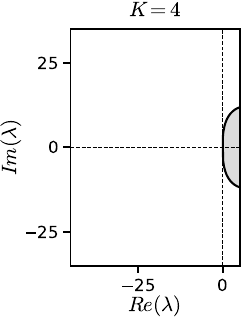}%
    \caption{\captionText{\LU}Note that the stability regions for
        \IE-SDC are very similar.}
\end{figure}

For comparison, we show the stability contours of the $LU$ preconditioner by Weiser~\cite{weiser2015faster} in Figure~\ref{fig:stab_LU}.
Stability for SDC with a standard implicit Euler sweeper is very similar and not shown.
Note that the \LU preconditioner is lower triangular and does not allow for parallelism in the sweep.
Since stability regions of \MINSRFLEX and \LU are similar, we can conclude that, with an optimized choice of coefficients, parallelism in SDC can be obtained without loss of stability.
For further comparison, we show the stability contours for the \HOUWENSOMMEIJER preconditioner, obtained by minimizing the spectral radius of $\Imat-\Qdi\Qmat$~ \cite[Sec.~4.3.4]{houwen1991iterated} in Figure~\ref{fig:stab_VDHS}.
For $K\in \{1, 2, 3\}$ the stability regions are very small compared to \LU or \MINSRFLEX. Even though the stable regions grows as $K$ increases, it does not include the imaginary axis, making the method unsuitable for oscillatory problems with purely imaginary eigenvalues.
Bounded stability regions are also observed for the MIN preconditioner~\cite{speck2018parallelizing} (not shown).
Since the spectral radius of $\Imat-\Qdi\Qmat$ is significantly larger for \texttt{MIN} than \MINSRS and \HOUWENSOMMEIJER, we do not consider it in the rest of this paper.

\begin{figure}
    \label{fig:stab_VDHS}
    \includegraphics[width=0.24\linewidth]{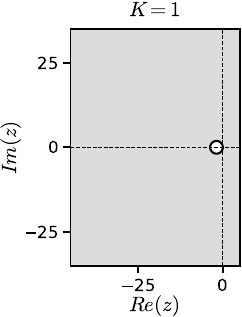}%
    \hfill%
    \includegraphics[width=0.24\linewidth]{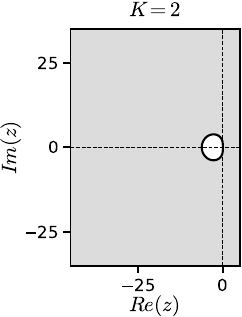}%
    \hfill%
    \includegraphics[width=0.24\linewidth]{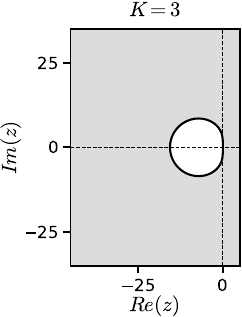}%
    \hfill%
    \includegraphics[width=0.24\linewidth]{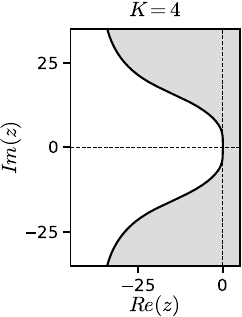}%
    \caption{\captionText{\HOUWENSOMMEIJER}}
\end{figure}

\begin{remark}
Experiments not documented here suggest that when minimizing the spectral radius of the stiff limit $\Imat-\Qdi\Qmat$, enforcing monotonically increasing coefficients for \MINSRS provides a notable
improvement in numerical stability.
\end{remark}

\section{Computational efficiency}
\label{sec:compEfficiency}

We compare computational cost versus accuracy of our three new SDC preconditioners against SDC preconditioners from the literature, the classical explicit $4^{th}$ order (RK4) and the $L$-stable stiffly accurate implicit RKM ESDIRK4(3)6L[2]SA~\cite{kennedy2016diagonally} (ESDIRK43).
The two RKM were implemented in \texttt{pySDC} by Baumann et al.~\cite{baumann2042pursing}.
Again all figures in this section can be reproduced using scripts in the provided code~\cite{speck2024pySDC}.

\subsection{Estimating computational cost}
\label{sec:costEstimation}
A fair run-time assessment of parallel SDC requires an optimized parallel implementation
which is the subject of a separate work~\cite{freese2024parallel}.
Here we instead estimate computational cost by considering the
elementary operations of each scheme.

To solve a system of ODEs~\eqref{eq:ODEsys},  both SDC and RKM need to
\begin{enumerate}
    \item evaluate the right-hand-side (RHS) $f(u, t)$ for given $u$, $t$ and
    \item solve the following non-linear system for some $b$, $t$ and $\alpha$
    \begin{equation}\label{eq:systemSolve}
        u - \alpha f(u, t) = b.
    \end{equation}
\end{enumerate}
We solve~\eqref{eq:systemSolve} with an exact Newton iteration, starting from the initial solution $u_0$ or the previous SDC iterate $u_m^{k-1}$.
We stop iterating when a set problem-dependent tolerance is reached or after $300$ iterations.
In our numerical experiments, the Jacobian $J_{f}$ can be computed and $I-\alpha J_{f}$ is inverted analytically.
In this case, the cost of one Newton iteration is similar to the cost of a
RHS evaluation.
For all methods, we therefore model that computational cost of a simulation by
$N_\text{Newton} + N_\text{RHS}$.
For the Allen-Cahn problem, since the cost of one Newton iteration is bigger than this
of one RHS evaluation, we model the computation cost in this case by $2N_\text{Newton} + N_\text{RHS}$.
Note that $N_{Newton}=0$ for explicit methods like RK4 or \PIC-SDC.

For parallel SDC, the computations in every sweep can be parallelized across $M$ threads.
To account for unavoidable overheads from communication or competition for resources between threads, we assume a parallel efficiency of $P_\text{eff}=0.8$ or $80 \%$ which is achievable in optimized implementations using compiled languages and \texttt{OpenMP}~\cite{freese2024parallel}.
Thus, in our performance model, we divide the computational cost
estimate for SDC by $M P_\text{eff}$ instead of $M$.
The cost model is validated below against wallclock time measurements for the Allen-Cahn problem using MPI parallelization.
Parallel efficiences there range between $74\%$ and $93\%$, depending on time step size, providing additional evidence that our assumption of $80\%$ efficiency is reasonable.
No runtimes are shown for the Lorenz or Prothero-Robinson problem but can be generated with the provided code
\cite{speck2024pySDC}.

\subsection{Lorenz system}
\label{sec:efficiencyLorenz}
We first consider the non-linear system of ODEs
\begin{equation}
    \frac{d x}{dt} = \sigma (y - x),\quad
    \frac{d y}{dt} = x(\rho-z)-y,\quad
    \frac{d z}{dt} = xy-\beta z,
\end{equation}
with $(\sigma, \rho, \beta)=(10, 28, 8/3)$.
The initial value is $(x(0), y(0), z(0)) = (5,-5,20)$ and we use a Newton tolerance of $10^{-12}$.
We set the final time for the simulation to $T=1.24$, which corresponds to two revolutions around one of the attraction points.
A reference solution is computed using an embedded RKM of order 5(4)~\cite{dormand1980family} implemented in \texttt{scipy} with an error tolerance of $10^{-14}$.

\begin{figure}
    \label{fig:lorenzError}
    \includegraphics[width=0.48\linewidth]{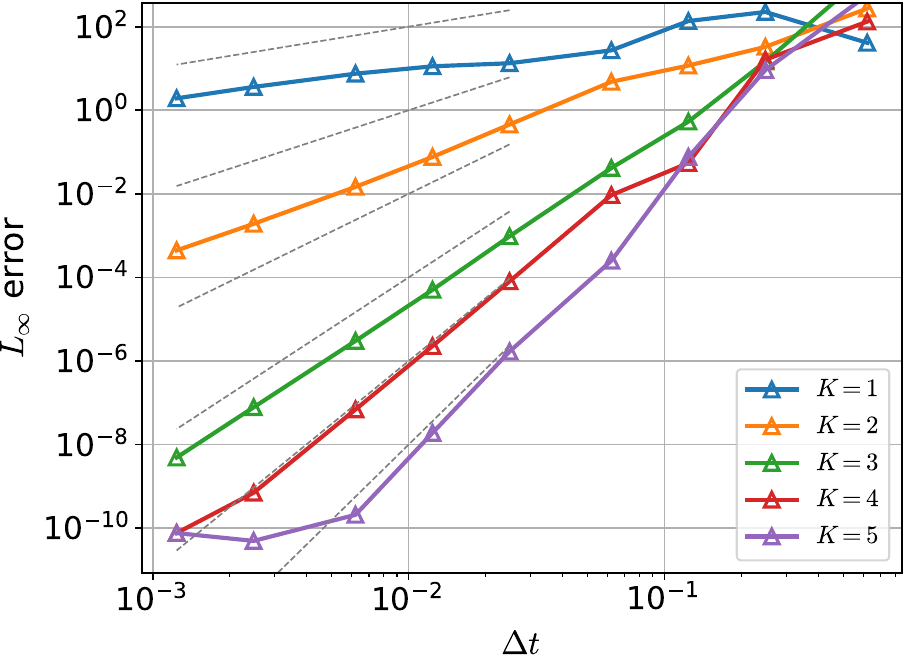}%
    \hfill%
    \includegraphics[width=0.48\linewidth]{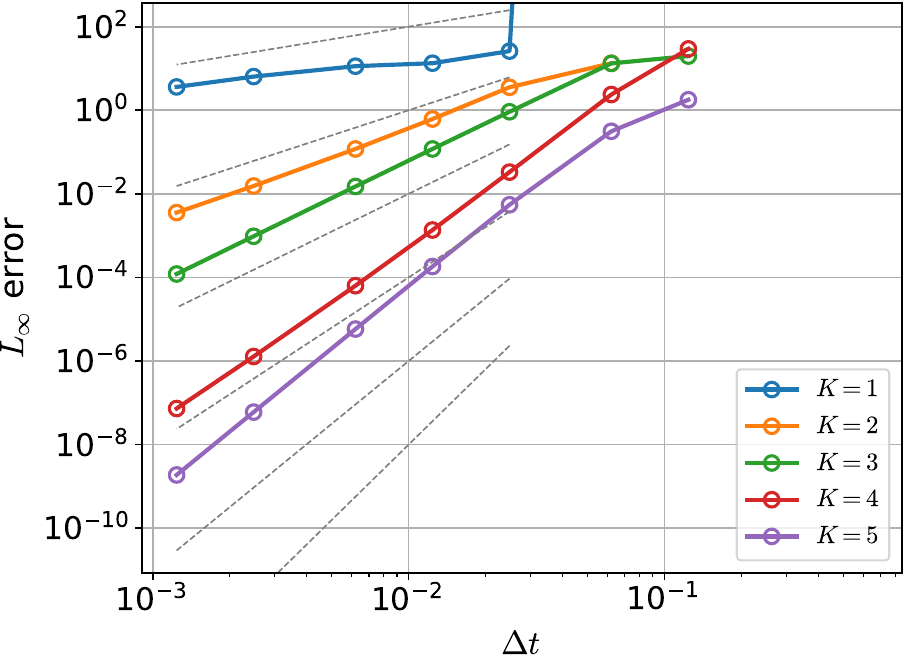}%
    \caption{Error vs.\ time step size for SDC for the Lorenz problem using $M=4$ \RadauRight nodes,
        for $K=1, \ldots, 5$ sweeps per time step.
        Left: \MINSRNS preconditioner, right: \PIC preconditioner.
        Dashed gray lines with slopes from 1 to 6 are shown as a guide to the eye.}
\end{figure}
Figure~\ref{fig:lorenzError} shows error versus time step size for \MINSRNS for SDC for $K = 1, \ldots, 5$.
Since the Lorenz system is not stiff, we compare against the Picard iteration (\PIC), which is known to be efficient for non-stiff problems.
We do see the expected order increase by one per iteration as well as the additional order gain for \MINSRNS starting at $K=3$ seen already for the Dahlquist problem.
For the same number of sweeps, this makes \MINSRNS more accurate than Picard iteration or other SDC methods.

\begin{figure}
    \label{fig:lorenzCost}
    \includegraphics[width=0.48\linewidth]{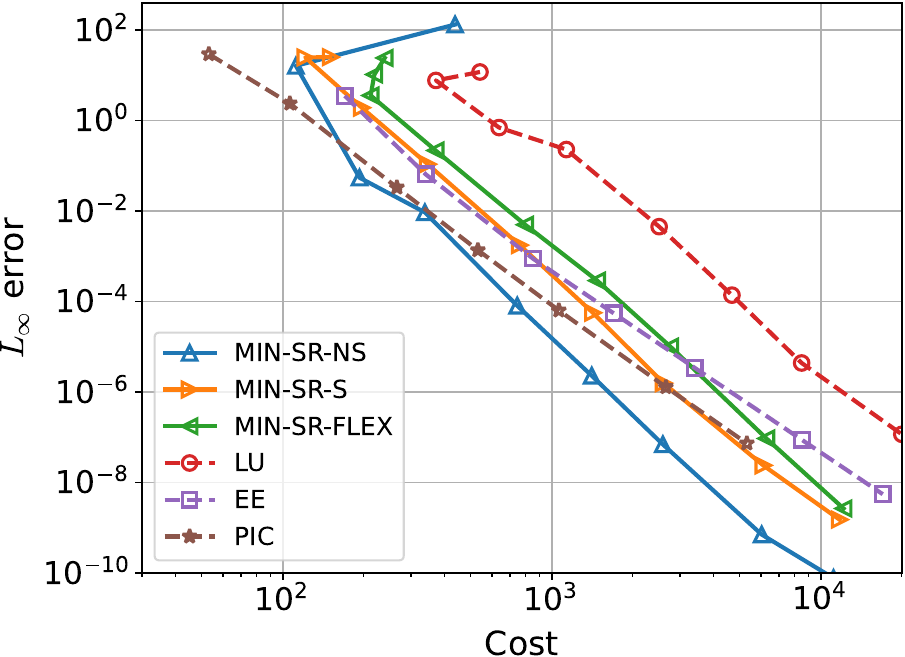}%
    \hfill%
    \includegraphics[width=0.48\linewidth]{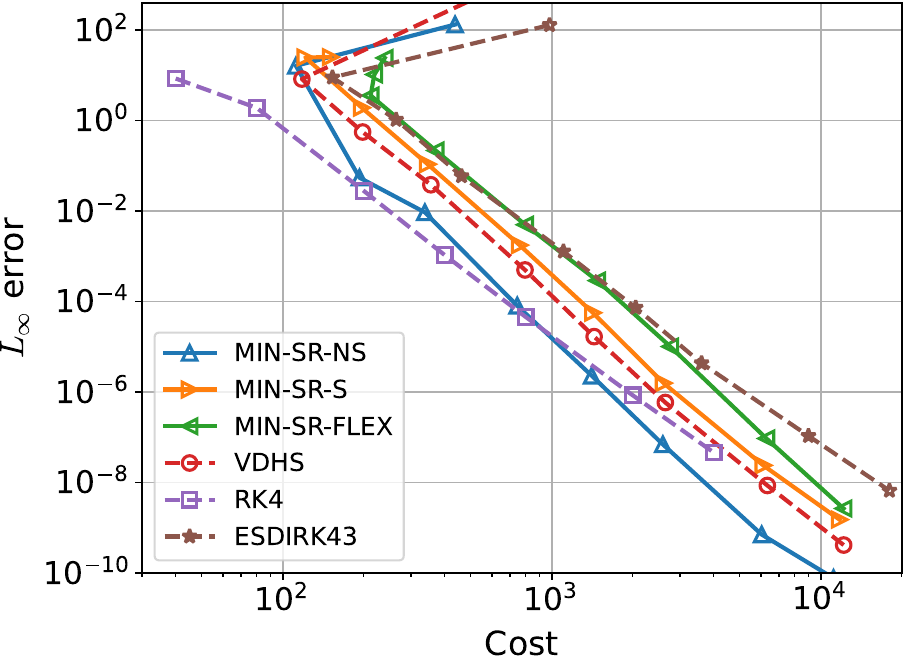}%
    \caption{Error vs.\ cost for SDC for the Lorenz problem using $M=4$ \RadauRight nodes and $K=4$ sweeps.
    Left: comparison with classical SDC preconditioners, right: comparison with efficient time integration methods from the literature and SDC with \HOUWENSOMMEIJER preconditioner.}
\end{figure}
Figure~\ref{fig:lorenzCost} shows error against modelled computational cost for \MINSRS and a variety of other SDC variants (left) as well as RKM and the \HOUWENSOMMEIJER preconditioner (right).
The parallel and PIC preconditioners significantly outperform classical SDC with explicit Euler or LU sweep.
As expected, the preconditioner for non-stiff problems \MINSRNS is the most efficient, although the stiff preconditioners \MINSRS and \MINSRFLEX remain competitive.
\MINSRNS also outperforms the \HOUWENSOMMEIJER preconditioner and ESDIRK43 RKM.
For errors above $10^{-4}$, explicit RKM4 is more efficient than \MINSRNS
but for errors below $10^{-6}$ \MINSRNS outperforms RKM4.
When increasing the number of sweeps to $K=5$, the advantage in efficiency of
\MINSRNS over RKM4 becomes more pronounced, see Figure~\ref{fig:lorenzCost2}.

\begin{figure}
    \label{fig:lorenzCost2}
    \centering
    \includegraphics[width=0.48\linewidth]{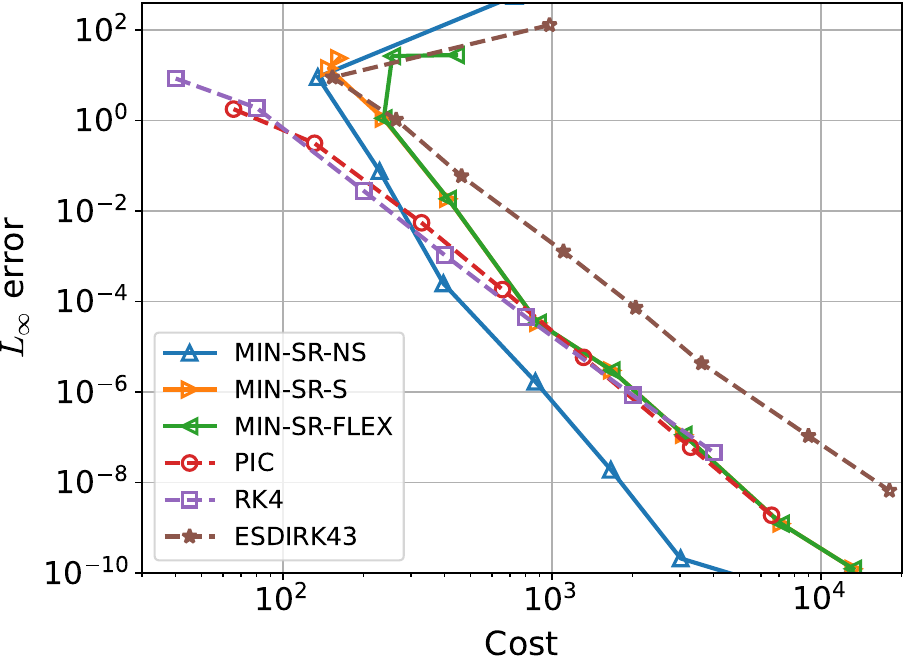}%
    \caption{Error vs.\ cost for SDC and Picard iteration for the Lorenz problem using $M=4$ \RadauRight nodes and $K=5$ sweeps in comparison to RKM4 and ESIDRK43.}
\end{figure}

\subsection{Prothero-Robinson problem}
\label{sec:protheroRobinson}

Our second test case is the stiff ODE by Prothero and Robinson~\cite{prothero1974stability}
\begin{equation}
    \label{eq:protheroRobinson}
    \frac{du}{dt} = \frac{u - g(t)}{\varepsilon} + \frac{dg}{dt}.
\end{equation}
The analytical solution for this ODE is $u(t)=g(t)$ and we set $g(t)=\cos(t)$ and $\varepsilon=10^{-3}$.
We also set a Newton tolerance of $10^{-12}$.

\begin{figure}
    \label{fig:protheroRobinsonK4}
    \includegraphics[width=0.48\linewidth]{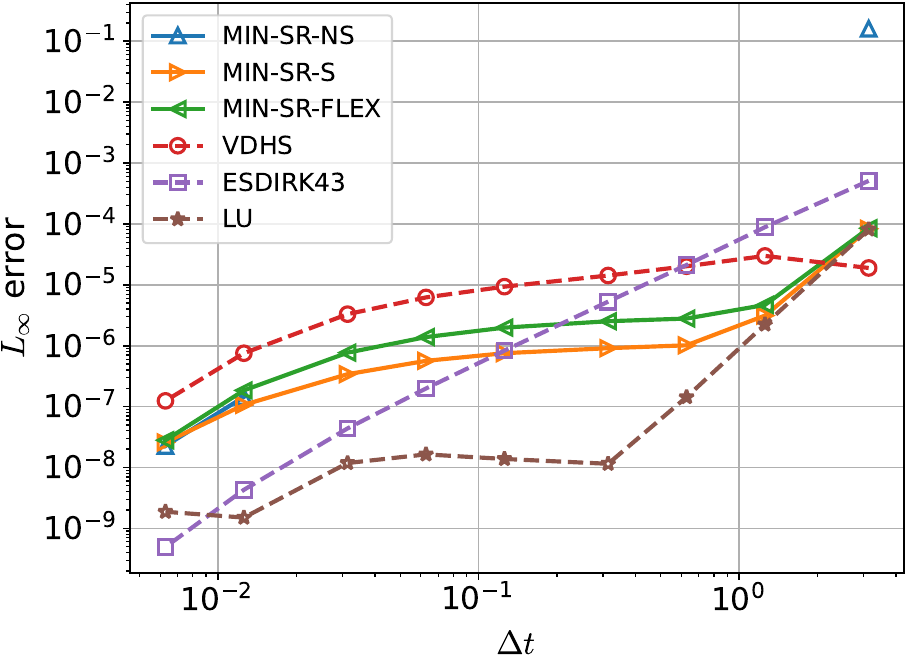}%
    \hfill%
    \includegraphics[width=0.48\linewidth]{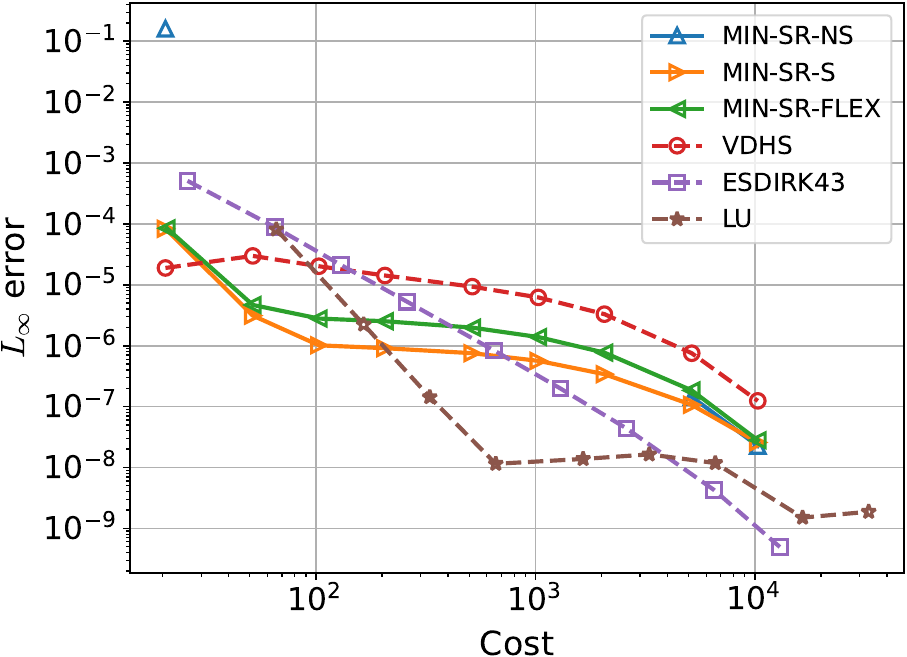}%
    \caption{Error vs.\ time step size for parallel SDC and classical time-integration schemes
    for the Prothero-Robinson problem with $\varepsilon=10^{-3}$.
    SDC uses $M=4$ \RadauRight nodes and $K=4$ sweeps.
    Left: error vs.\ time step, right: error vs.\ cost.}
\end{figure}
Figure~\ref{fig:protheroRobinsonK4} shows error versus time step size (left) and modelled computational cost (right) for our three parallel SDC methods, non-parallel \LU-SDC, parallel \HOUWENSOMMEIJER SDC and the implicit ESDIRK43 RKM~\cite{kennedy2016diagonally}.
All SDC variants use $K=4$ sweeps.
The parallel SDC variants all show a noticable range where the error does not decrease with time step size.
This is a known phenomenon for the Prothero-Robinson problem~\cite[Sec.~6.1]{weiser2015faster} and means that a very small time step is required to recover the theoretically expected convergence order.
While \MINSRS outperforms \LU-SDC in efficiency up to an error of around $10^{-6}$, its stalling convergence after results in \LU-SDC being more efficient for very high accuracies.

\begin{figure}
    \label{fig:protheroRobinsonK6}
    \includegraphics[width=0.48\linewidth]{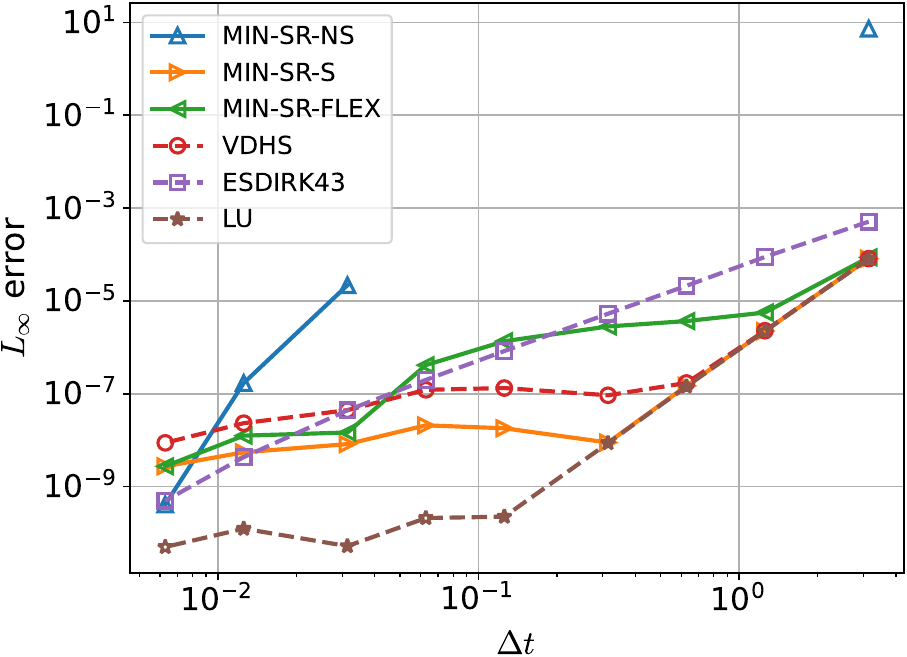}%
    \hfill%
    \includegraphics[width=0.48\linewidth]{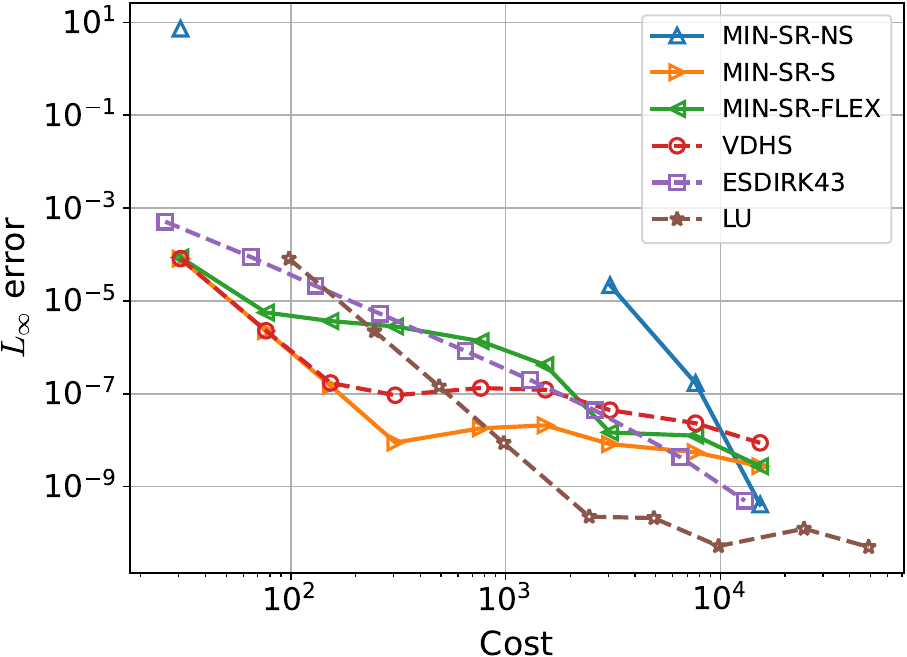}%
    \caption{Comparison of diagonal SDC and classical time-integration schemes
        on the Prothero-Robinson problem, using $\varepsilon=10^{-3}$.
        Each SDC configuration uses $M=4$ \RadauRight nodes and $K=6$ sweeps.
        Left : error vs.\ time step, right: error vs.\ cost.}
\end{figure}
If we increase the number of sweeps to $K=6$, the error level at which convergence stalls is reduced, see Figure~\ref{fig:protheroRobinsonK6}.
This also pushes down the error where \LU-SDC becomes more efficient than \MINSRS to around $10^{-8}$.
Nevertheless, in both cases \MINSRS will be an attractive integrator unless extremely tight accuracy is required.

\subsection{Allen-Cahn equation}
\label{sec:allenCahn}

As last test problem, we consider the one-dimen\-sional Allen-Cahn equation with driving force
\begin{equation}
    \frac{\partial u}{\partial t}
    = \frac{\partial^2 u}{\partial x^2}
        - \frac{2}{\varepsilon^2} u (1 - u) (1 - 2u) - 6 d_w u (1 - u),
    \quad x \in [-0.5, 0.5], \quad t \in [0, T].
\end{equation}
using inhomogeneous Dirichlet boundary conditions.
The exact solution is
\begin{equation}
    \label{eq:allenCahnSolution}
    u(x, t)= 0.5 \left[1 + \tanh\left(\frac{x - vt}{\sqrt{2}\varepsilon}\right)\right],
    \quad v = 3 \sqrt{2} \varepsilon d_w,
\end{equation}
and we use it to set the initial solution $u(x,0)$ and the boundary conditions
for $x= \pm 0.5$.
We set $T=50$ as simulation interval and parameters $\varepsilon=d_w=0.04$.
The spatial derivative is discretized with a second order finite-difference scheme on $2047$ grid points and we solve~\eqref{eq:systemSolve} in each Newton iteration with a
sparse linear solver from \texttt{scipy}.
Numerical experiments were run on one compute node of the \texttt{JUSUF} cluster (AMD EPYC 7742 2.25 GHz)
at Jülich Supercomputing Center, using a locally compiled version
of \texttt{Python=3.11.9} (\texttt{GCC=12.3.0}),
with \texttt{numpy=2.0.2}, \texttt{scipy=1.14.1},
and \texttt{mpi4py=4.0.0} wrapping \texttt{ParaStationMPI=5.9.2-1}.

Figure~\ref{fig:allenCahnAccuracy} shows the absolute error in the $L_2$ norm 
at $T$ versus the time step size. 
All methods converge to an error of around $2 \times 10^{-4}$, which corresponds
to the space discretization error for the chosen grid. 
The fastest converging method is SDC with LU, followed by ESDIRK43 and SDC with MIN-SR-FLEX preconditioners.

\begin{figure}[th]
    \label{fig:allenCahnAccuracy}
    \centering
    \includegraphics[width=0.48\linewidth]{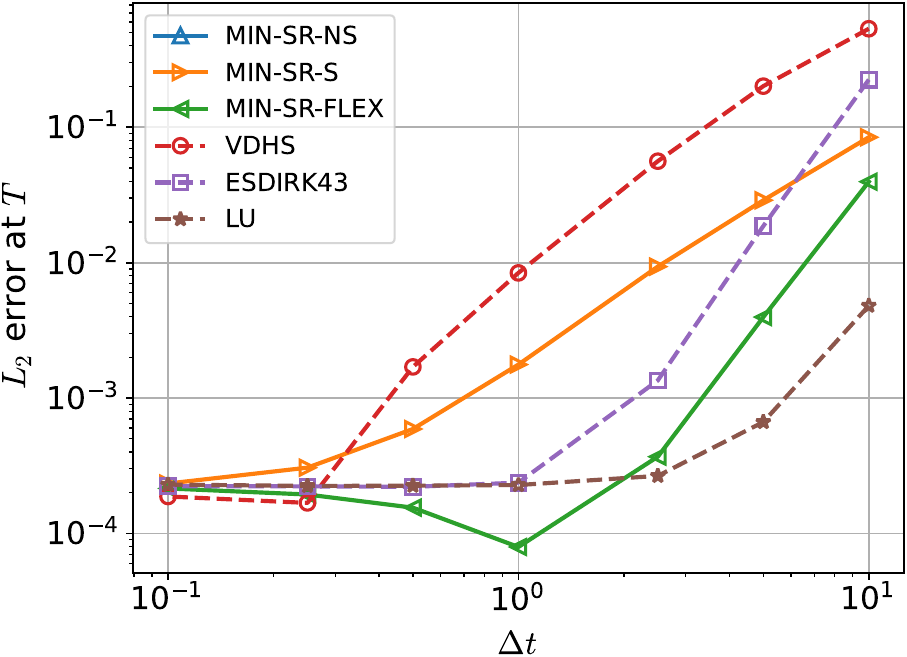}%
    \caption{Comparison of the accuracy of the diagonal SDC preconditioners and classical time-integration schemes for the Allen-Cahn equation. Each SDC configuration uses $M=4$ \RadauRight nodes and $K=4$ sweeps.}
\end{figure}

Figure~\ref{fig:allenCahnEfficiency} (left) shows the absolute error in the $L_2$ norm at $T$ versus modelled computational cost.
All methods converge to an error of around $2\cdot 10^{-4}$, which corresponds to the space discretization error for the chosen grid.
The most efficient methods is \MINSRFLEX parallel SDC, followed by ESDIRK43 and SDC with LU preconditioner.
 Figure~\ref{fig:allenCahnEfficiency} (right) shows the absolute $L_2$ error versus wall-clock time in seconds.
The plot looks very similar when using modelled cost, illustrating that our cost model gives a reasonable approximation of actual cost.
Again, \MINSRFLEX is the most efficient integrator except for very loose error tolerances of $10^{-1}$ and above where \MINSRS is faster.
Note that for SDC with \LU preconditioner to catch up to \MINSRFLEX in terms of performance for a given time step size, the efficiency of the parallel implementation would need to be
below $40\%$.
\begin{figure}[th]
    \label{fig:allenCahnEfficiency}
    \includegraphics[width=0.48\linewidth]{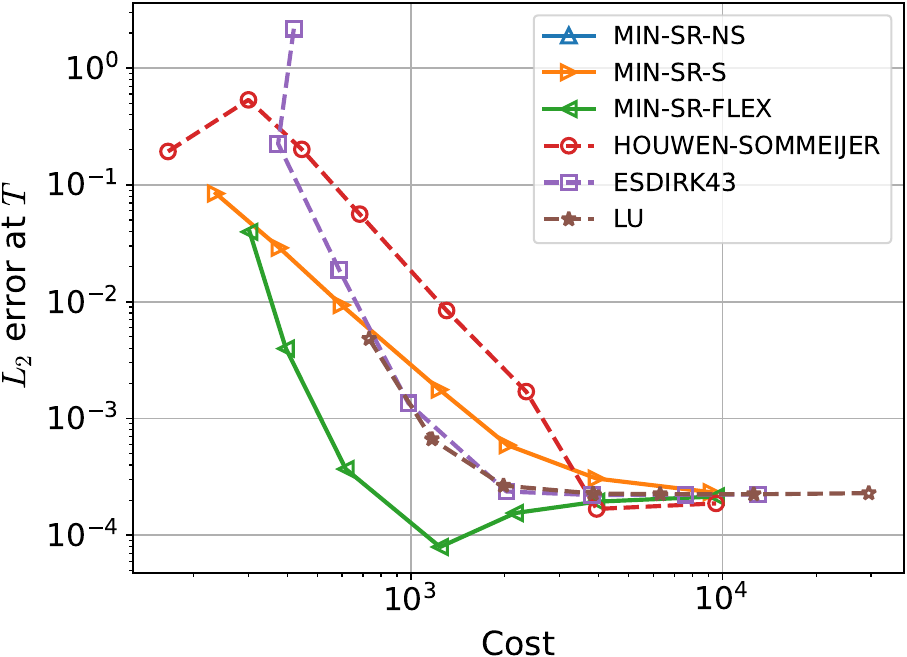}%
    \hfill%
    \includegraphics[width=0.48\linewidth]{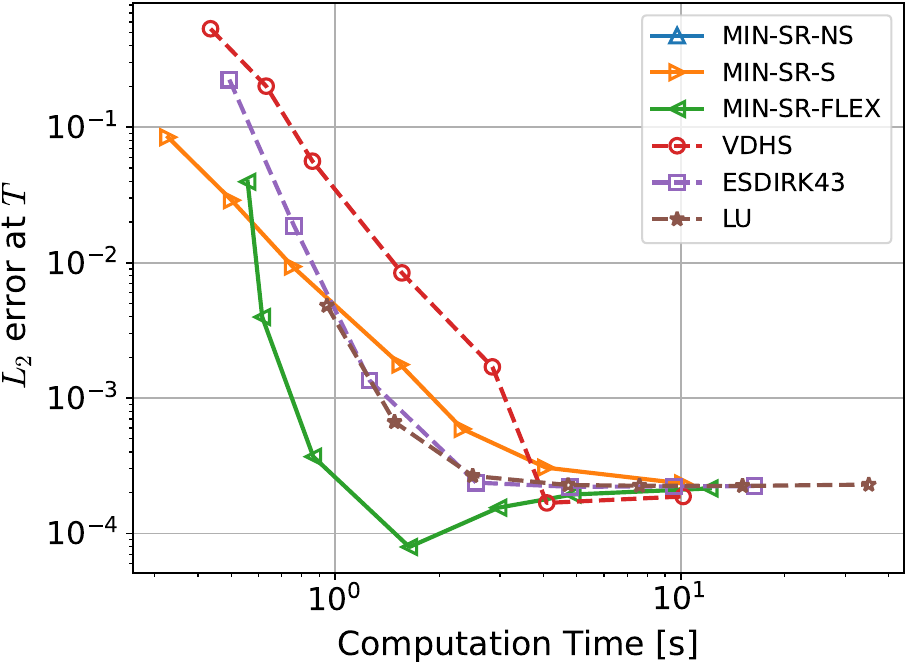}%
    \caption{Comparison of the efficiency of diagonal SDC preconditioners and classical time-integration schemes
        for the Allen-Cahn equation.
        Each SDC configuration uses $M=4$ \RadauRight nodes and $K=4$ sweeps.
        Left : error vs.\ modelled computational cost, right: error vs. computation time.}
\end{figure}

\section{Conclusions}
\label{sec:conclusions}
Using a diagonal preconditioner for spectral deferred corrections allows to exploit small-scale parallelization in time for a number of threads up to the number of quadrature nodes in the underlying collocation formula.
However, efficiency and stability of parallel SDC depends critically on the coefficients in the preconditioner.
We consider two minimization problems, one for stiff and one for non-stiff time-dependent
problems, that can be solved to find optimized parameter.
This allows us to propose three new sets of coefficients, \MINSRNS for non-stiff problems and
\MINSRS and \MINSRFLEX for stiff problems.
While we use numerical optimization to determine the \MINSRS diagonal coefficients, 
we can obtain the \MINSRNS and \MINSRFLEX coefficients analytically.
\MINSRFLEX is a non-stationary iteration where the preconditioner changes in every 
SDC sweep and is designed to generate a nilpotent error propagation matrix.

We demonstrate by numerical experiments that the three new variants of parallel SDC increase their order by at least one per sweep, similar to existing non-parallel SDC methods.
The variants designed for stiff problems have excellent stability properties and might well be $A$-stable, although we do not have a rigorous proof yet.
Stability regions of \MINSRFLEX in particular are very similar to those of the non-parallel \LU-SDC variant and much larger than those of the Iterated Implicit Runge-Kutta methods by Van der Houwen~\cite{houwen1991iterated}.
To model computational efficiency, we count the number of right hand side evaluations and Newton iterations for all schemes and assume $80\%$ parallel efficiency for the implementation of parallel SDC.
We compare error against computational effort of our three new parallel SDC variants against parallel SDC variants from the literature, explicit RKM4 and implicit ESDIRK43 for the non-stiff Lorenz system, the stiff Prothero-Robinson problem and the stiff, one-dimensional Allen-Cahn equation.
For all three test problems, the new parallel SDC variants can outperform existing parallel SDC variants, state-of-the-art serial SDC methods as well as RKM4 and ESDIRK43.
Wallclock time measurements for the Allen-Cahn problem show that our cost model closely tracks actual runtimes and that a parallel implementation of SDC can be more efficient than optimized serial SDC methods as well as Runge-Kutta methods.

\appendix
\section{Optimized diagonal coefficients for the stiff limit}
\label{ap:coeffs}

We repeat here the coefficients used for our numerical experiments, either numerically computed or obtained from the literature as well as the resulting spectral radius of the $\KS$ matrix.
The values are for $M=4$ \RadauRight nodes from a Legendre distribution.
\begin{align*}
    &\text{\HOUWENSOMMEIJER by Van der Houwen and Sommeijer~\cite{houwen1991iterated}}\\
        &\quad\vect{d}=[0.32049937,\; 0.08915379,\; 0.18173956,\; 0.2333628],
         \quad \rho(\KS)=0.025\\
    &\text{\texttt{MIN} by Speck~\cite{speck2018parallelizing}}\\
        &\quad\vect{d}=[0.17534868,\; 0.0619158,\; 0.1381934,\; 0.19617814],
         \quad \rho(\KS)=0.42\\
    &\text{\texttt{MIN3} by Speck et al.~\cite{speck2024pySDC}}\\
        &\quad\vect{d}=[0.31987868,\; 0.08887606,\; 0.18123663,\; 0.23273925],
        \quad \rho(\KS)=0.0081\\
    &\text{\MINSRS introduced in this paper}\\
        &\quad\vect{d}=[0.05363588,\; 0.18297728,\; 0.31493338,\; 0.38516736],
        \quad \rho(\KS)=0.00024
\end{align*}

%\section*{Acknowledgments}
%\todo{Who do we thank? Martin S beautified some proofs ...
%    nah, let's just thanks KIT for paying for Gaya's trip to HH. Can we say something like "Thanks KIT for funding the travels to Hamburg, %rare ones that I got the refund for?"}

\bibliographystyle{siamplain}
\bibliography{paper}

\begin{thebibliography}{10}

\bibitem{baumann2042pursing}
{\sc T.~Baumann, S.~Götschel, T.~Lunet, D.~Ruprecht, and R.~Speck}, {\em
  Adaptive time step selection for spectral deferred corrections}, 2024.
\newblock Submitted.

\bibitem{Burrage1993}
{\sc K.~Burrage}, {\em {Parallel methods for initial value problems}}, Applied
  Numerical Mathematics, 11 (1993), pp.~5--25,
  \url{https://doi.org/10.1016/0168-9274(93)90037-R}.

\bibitem{Butcher1976}
{\sc J.~C. Butcher}, {\em {On the implementation of implicit Runge--Kutta
  methods}}, {BIT Numerical Mathematics}, 16 (1976), pp.~237--240.

\bibitem{GayaThesis}
{\sc G.~Caklovic}, {\em {ParaDiag and Collocation Methods: Theory and
  Implementation}}, PhD thesis, Karlsruher Institut für Technologie (KIT),
  2023, \url{https://doi.org/10.5445/IR/1000164518}.

\bibitem{christlieb2009comments}
{\sc A.~Christlieb, B.~Ong, and J.-M. Qiu}, {\em {Comments on high-order
  integrators embedded within integral deferred correction methods}},
  Communications in Applied Mathematics and Computational Science, 4 (2009),
  pp.~27--56, \url{https://doi.org/10.2140/camcos.2009.4.27}.

\bibitem{dormand1980family}
{\sc J.~R. Dormand and P.~J. Prince}, {\em {A family of embedded Runge-Kutta
  formulae}}, Journal of computational and applied mathematics, 6 (1980),
  pp.~19--26.

\bibitem{dutt2000spectral}
{\sc A.~Dutt, L.~Greengard, and V.~Rokhlin}, {\em {Spectral Deferred Correction
  Methods for Ordinary Differential Equations}}, BIT Numerical Mathematics, 40
  (2000), pp.~241--266, \url{https://doi.org/10.1023/A:1022338906936}.

\bibitem{EmmettMinion2012}
{\sc M.~Emmett and M.~L. Minion}, {\em {Toward an Efficient Parallel in Time
  Method for Partial Differential Equations}}, Communications in Applied
  Mathematics and Computational Science, 7 (2012), pp.~105--132,
  \url{https://doi.org/10.2140/camcos.2012.7.105},
  \url{http://dx.doi.org/10.2140/camcos.2012.7.105}.

\bibitem{FalgoutEtAl2014_MGRIT}
{\sc R.~D. Falgout, S.~Friedhoff, T.~V. Kolev, S.~P. MacLachlan, and J.~B.
  Schroder}, {\em {Parallel time integration with multigrid}}, SIAM Journal on
  Scientific Computing, 36 (2014), pp.~C635--C661,
  \url{https://doi.org/10.1137/130944230},
  \url{http://dx.doi.org/10.1137/130944230}.

\bibitem{freese2024parallel}
{\sc P.~Freese, S.~G{\"o}tschel, T.~Lunet, D.~Ruprecht, and M.~Schreiber}, {\em
  Parallel performance of shared memory parallel spectral deferred
  corrections}, arXiv preprint arXiv:2403.20135,  (2024).

\bibitem{Gander2015_Review}
{\sc M.~J. Gander}, {\em {50 years of Time Parallel Time Integration}}, in
  Multiple Shooting and Time Domain Decomposition, Springer, 2015,
  \url{https://doi.org/10.1007/978-3-319-23321-5_3},
  \url{http://dx.doi.org/10.1007/978-3-319-23321-5_3}.

\bibitem{gautschi2004orthogonal}
{\sc W.~Gautschi}, {\em Orthogonal polynomials: computation and approximation},
  OUP Oxford, 2004.

\bibitem{Gear1988}
{\sc C.~Gear}, {\em {Parallel Methods for Ordinary Differential Equations}},
  CALCOLO, 25 (1988), pp.~1--20.

\bibitem{guibert2007parallel}
{\sc D.~Guibert and D.~Tromeur-Dervout}, {\em {Parallel deferred correction
  method for CFD problems}}, in Parallel Computational Fluid Dynamics 2006,
  Elsevier, 2007, pp.~131--138.

\bibitem{hairer1993nonStiff}
{\sc E.~Hairer, S.~P. N{\o}rsett, and G.~Wanner}, {\em {Solving Ordinary
  Differential Equations {I}: Nonstiff problems}}, Springer-Verlag Berlin
  Heidelberg, 2nd~ed., 1993, \url{https://doi.org/10.1007/978-3-540-78862-1}.

\bibitem{houwen1992embedded}
{\sc P.~J. V.~D. Houwen, B.~P. Sommeijer, and W.~Couzy}, {\em {Embedded
  Diagonally Implicit Runge-Kutta Algorithms on Parallel Computers}},
  Mathematics of Computation, 58 (1992), pp.~135--159.

\bibitem{huang2006accelerating}
{\sc J.~Huang, J.~Jia, and M.~Minion}, {\em {Accelerating the convergence of
  spectral deferred correction methods}}, Journal of Computational Physics, 214
  (2006), pp.~633--656, \url{https://doi.org/10.1016/j.jcp.2005.10.004}.

\bibitem{IserlesNorsett1990}
{\sc A.~Iserles and S.~P. N{\o}rsett}, {\em {On the theory of parallel
  {R}unge-{K}utta methods}}, IMA Journal of Numerical Analysis, 10 (1990),
  pp.~463--488, \url{https://doi.org/10.1093/imanum/10.4.463}.

\bibitem{Jackson1991}
{\sc K.~R. Jackson}, {\em {A survey of parallel numerical methods for initial
  value problems for ordinary differential equations}}, IEEE Transactions on
  Magnetics, 27 (1991), pp.~3792--3797,
  \url{https://doi.org/10.1109/20.104928}.

\bibitem{JacksonEtAl1995}
{\sc K.~R. Jackson and S.~P. N{\o}rsett}, {\em {The Potential for Parallelism
  in Runge–Kutta Methods. Part 1: RK Formulas in Standard Form}}, SIAM
  Journal on Numerical Analysis, 32 (1995), pp.~49--82,
  \url{https://doi.org/10.1137/0732002}.

\bibitem{kennedy2016diagonally}
{\sc C.~A. Kennedy and M.~H. Carpenter}, {\em {Diagonally implicit Runge-Kutta
  methods for ordinary differential equations. A review}}, tech. report, 2016.

\bibitem{LevequeEtAl2023}
{\sc S.~Leveque, L.~Bergamaschi, A.~Martinez, and J.~W. Pearson}, {\em {Fast
  Iterative Solver for the All-at-Once Runge--Kutta Discretization}}.
\newblock 2023.

\bibitem{Lie1987}
{\sc I.~Lie}, {\em {Some aspects of parallel Runge-Kutta methods}}, tech.
  report, Trondheim TU. Inst. Math., Trondheim, 1987,
  \url{https://cds.cern.ch/record/201368}.

\bibitem{LionsEtAl2001}
{\sc J.-L. Lions, Y.~Maday, and G.~Turinici}, {\em {A "parareal" in time
  discretization of {PDE}'s}}, Comptes Rendus de l'Académie des Sciences -
  Series I - Mathematics, 332 (2001), pp.~661--668,
  \url{https://doi.org/10.1016/S0764-4442(00)01793-6},
  \url{http://dx.doi.org/10.1016/S0764-4442(00)01793-6}.

\bibitem{MunchEtAl2023}
{\sc P.~Munch, I.~Dravins, M.~Kronbichler, and M.~Neytcheva}, {\em
  {Stage-Parallel Fully Implicit Runge{\textendash}Kutta Implementations with
  Optimal Multilevel Preconditioners at the Scaling Limit}}, {SIAM} Journal on
  Scientific Computing,  (2023), pp.~S71--S96,
  \url{https://doi.org/10.1137/22m1503270}.

\bibitem{nievergelt1964parallel}
{\sc J.~Nievergelt}, {\em {Parallel methods for integrating ordinary
  differential equations}}, Communications of the ACM, 7 (1964), pp.~731--733.

\bibitem{NorsettEtAl1989}
{\sc S.~P. N{\o}rsett and H.~H. Simonsen}, {\em {Aspects of parallel
  Runge-Kutta methods}}, in Numerical Methods for Ordinary Differential
  Equations, A.~Bellen, C.~W. Gear, and E.~Russo, eds., Berlin, Heidelberg,
  1989, Springer Berlin Heidelberg, pp.~103--117.

\bibitem{OngEtAl2016}
{\sc B.~W. Ong, R.~D. Haynes, and K.~Ladd}, {\em {Algorithm 965: {RIDC}
  Methods: A Family of Parallel Time Integrators}}, ACM Trans. Math. Softw., 43
  (2016), pp.~8:1--8:13, \url{https://doi.org/10.1145/2964377},
  \url{http://dx.doi.org/10.1145/2964377}.

\bibitem{OngEtAl2020}
{\sc B.~W. Ong and R.~J. Spiteri}, {\em {Deferred Correction Methods for
  Ordinary Differential Equations}}, Journal of Scientific Computing, 83
  (2020), \url{https://doi.org/10.1007/s10915-020-01235-8}.

\bibitem{OREL1993241}
{\sc B.~Orel}, {\em {Parallel Runge–Kutta methods with real eigenvalues}},
  Applied Numerical Mathematics, 11 (1993), pp.~241--250,
  \url{https://doi.org/https://doi.org/10.1016/0168-9274(93)90051-R}.

\bibitem{Pazner2017700}
{\sc W.~Pazner and P.-O. Persson}, {\em {Stage-parallel fully implicit
  Runge–Kutta solvers for discontinuous Galerkin fluid simulations}}, Journal
  of Computational Physics, 335 (2017), pp.~700--717,
  \url{https://doi.org/10.1016/j.jcp.2017.01.050}.

\bibitem{PETCU2001}
{\sc D.~Petcu}, {\em {Experiments with an ODE solver on a multiprocessor
  system}}, Computers \& Mathematics with Applications, 42 (2001),
  pp.~1189--1199.

\bibitem{prothero1974stability}
{\sc A.~Prothero and A.~Robinson}, {\em {On the stability and accuracy of
  one-step methods for solving stiff systems of ordinary differential
  equations}}, Mathematics of Computation, 28 (1974), pp.~145--162.

\bibitem{qu2016numerical}
{\sc W.~Qu, N.~Brandon, D.~Chen, J.~Huang, and T.~Kress}, {\em A numerical
  framework for integrating deferred correction methods to solve high order
  collocation formulations of odes}, Journal of Scientific Computing, 68
  (2016), pp.~484--520.

\bibitem{RauberEtAl1996}
{\sc T.~Rauber and G.~Rünger}, {\em {Parallel Implementations of Iterated
  Runge-Kutta Methods}}, The International Journal of Supercomputer
  Applications and High Performance Computing, 10 (1996), pp.~62--90,
  \url{https://doi.org/10.1177/109434209601000103}.

\bibitem{ruprecht2016spectral}
{\sc D.~Ruprecht and R.~Speck}, {\em {Spectral deferred corrections with
  fast-wave slow-wave splitting}}, SIAM Journal on Scientific Computing, 38
  (2016), pp.~A2535--A2557.

\bibitem{SchoebelEtAl2020}
{\sc R.~Schöbel and R.~Speck}, {\em {PFASST}-{ER}: combining the parallel full
  approximation scheme in space and time with parallelization across the
  method}, Computing and Visualization in Science, 23 (2020),
  \url{https://doi.org/10.1007/s00791-020-00330-5},
  \url{https://doi.org/10.1007/s00791-020-00330-5}.

\bibitem{SolodushkinEtAl2016}
{\sc S.~I. Solodushkin and I.~F. Iumanova}, {\em {Parallel numerical methods
  for ordinary differential equations: a survey}}, in CEUR Workshop
  Proceedings, vol.~1729, CEUR-WS, 2016, pp.~1--10.

\bibitem{Sommeijer1993}
{\sc B.~Sommeijer}, {\em {Parallel-iterated Runge-Kutta methods for stiff
  ordinary differential equations}}, {Journal of Computational and Applied
  Mathematics}, 45 (1993), pp.~151--168,
  \url{https://doi.org/10.1016/0377-0427(93)90271-C}.

\bibitem{speck2018parallelizing}
{\sc R.~Speck}, {\em {Parallelizing spectral deferred corrections across the
  method}}, {Comput. Visual Sci.}, 19 (2018), pp.~75--83,
  \url{https://doi.org/https://doi.org/10.1007/s00791-018-0298-x}.

\bibitem{speck2024pySDC}
{\sc R.~Speck, T.~Lunet, T.~Baumann, L.~Wimmer, and I.~Akramov}, {\em
  Parallel-in-time/pysdc}, Mar. 2024,
  \url{https://doi.org/10.5281/zenodo.13828395},
  \url{https://doi.org/10.5281/zenodo.13828395}.

\bibitem{VanderHouwen1993}
{\sc P.~van~der Houwen and B.~Sommeijer}, {\em {Analysis of parallel diagonally
  implicit iteration of Runge-Kutta methods}}, Applied Numerical Mathematics,
  11 (1993), pp.~169--188, \url{https://doi.org/10.1016/0168-9274(93)90047-U}.

\bibitem{VANDERHOUWEN1995309}
{\sc P.~{van der Houwen}, B.~Sommeijer, and W.~{van der Veen}}, {\em {Parallel
  iteration across the steps of high-order Runge-Kutta methods for nonstiff
  initial value problems}}, Journal of Computational and Applied Mathematics,
  60 (1995), pp.~309--329,
  \url{https://doi.org/https://doi.org/10.1016/0377-0427(94)00047-5}.

\bibitem{VanderHouwen1990}
{\sc P.~J. Van Der~Houwen and B.~P. Sommeijer}, {\em {Parallel iteration of
  high-order Runge-Kutta methods with stepsize control}}, Journal of
  Computational and Applied Mathematics, 29 (1990), pp.~111--127,
  \url{https://doi.org/10.1016/0377-0427(90)90200-J}.

\bibitem{houwen1991iterated}
{\sc P.~J. van~der Houwen and B.~P. Sommeijer}, {\em {Iterated Runge--Kutta
  Methods on Parallel Computers}}, SIAM Journal on Scientific and Statistical
  Computing, 12 (1991), pp.~1000--1028, \url{https://doi.org/10.1137/0912054}.

\bibitem{VanderHouwen1994}
{\sc P.~J. van~der Houwen, B.~P. Sommeijer, and W.~Van~der Veen}, {\em
  {Parallelism across the steps in iterated Runge-Kutta methods for stiff
  initial value problems}}, Numerical Algorithms, 8 (1994), pp.~293--312.

\bibitem{VanderveenEtAl1995}
{\sc W.~{van der Veen}, J.~{de Swart}, and P.~{van der Houwen}}, {\em
  {Convergence aspects of step-parallel iteration of Runge-Kutta methods}},
  Applied Numerical Mathematics, 18 (1995), pp.~397--411.

\bibitem{weiser2015faster}
{\sc M.~Weiser}, {\em {Faster SDC convergence on non-equidistant grids by DIRK
  sweeps}}, BIT Numerical Mathematics, 55 (2015), pp.~1219--1241,
  \url{https://doi.org/10.1007/s10543-014-0540-y}.

\end{thebibliography}

\end{document}